\documentclass[10pt]{amsart}

\usepackage{amssymb,amsthm,amstext,amsfonts,amsmath}
\usepackage[inline]{enumitem}

\newtheorem{claim}{Claim}

\newtheoremstyle{theorem}
     {11pt}
     {11pt}
     {}
     {}
     {\bfseries}
     {}
     {.5em}
     {\noindent\thmnumber{#2}. \thmname{#1}\thmnote{#3}}

\theoremstyle{theorem}

\newcommand{\CDH}{\ensuremath{\mathsf{CDH}}}
\newcommand{\csf}{{}\sp\omega{2}}
\newcommand{\homeo}[1]{\mathcal{H}(#1)}
\newcommand{\G}{\mathcal{G}}
\newcommand{\gen}[1]{\langle\!\langle #1\rangle\!\rangle}
\newcommand{\id}{\mathbf{1}}
\newcommand{\fix}[1]{\mathsf{fix}(#1)}
\newcommand{\C}{\mathcal{C}}
\newcommand{\U}{\mathcal{U}}
\newcommand{\co}[1]{\mathcal{CO}(#1)}
\newcommand{\inte}[2][X]{\mathrm{int}_{#1}({#2})}
\newcommand{\bd}[2][X]{\mathrm{bd}_{#1}({#2})}
\newcommand{\cl}[2][X]{\mathrm{cl}_{#1}({#2})}
\newcommand{\pair}[1]{\langle#1\rangle}

\newcommand{\V}{\mathcal{V}}

\newtheorem{thm}{Theorem}[section]
\newtheorem{lemma}[thm]{Lemma}
\newtheorem{propo}[thm]{Proposition}
\newtheorem{coro}[thm]{Corollary}

\newtheorem{ques}[thm]{Question}

\newtheorem{sclaim}[thm]{\empty\hskip-8pt}

\title{Countable dense homogeneity and the Cantor set}
\author[Hern\'andez-Guti\'errez]{Rodrigo Hern\'andez-Guti\'errez}
\address[Hern\'andez-Guti\'errez]{Departamento de Matem\'aticas, Universidad Aut\'onoma Metropolitana campus Iztapalapa, Av. San Rafael Atlixco 186, Col. Vicentina, Iztapalapa, 09340, Mexico city, Mexico}
\email[Hern\'andez-Guti\'errez]{rodrigo.hdz@gmail.com}

\date{\today}


\begin{document}

\begin{abstract}
It is shown that CH implies the existence of a compact Hausdorff space that is countable dense homogeneous, crowded and does not contain topological copies of the Cantor set. This contrasts with a previous result by the author which says that for any crowded Hausdorff space $X$ of countable $\pi$-weight, if ${}\sp\omega{X}$ is countable dense homogeneous, then $X$ must contain a topological copy of the Cantor set.
\end{abstract}
 
 \maketitle

\section{Introduction}

All the spaces considered below are Hausdorff spaces.\vskip6pt

A space $X$ is countable dense homogeneous (\CDH{}, henceforth) if $X$ is separable and every time $D,E\subset X$ are countable dense subsets, there is a homeomorphism $h:X\to X$ such that $h[D]=E$. Among examples of \CDH{} spaces we have the Euclidean spaces, the Hilbert cube and the Cantor set. For updated surveys on \CDH{} spaces, see sections 14, 15 and 16 of \cite{arh-vm-homogeneity} and the more recent \cite{hru-vm-cdh_survey}.

For some time there were no known ZFC examples of \CDH{} spaces that are compact and non-metrizable. In \cite{steprans-zhou-CDH} and \cite{cdhdefinable} it was shown that ${}\sp{\kappa}{2}$ is \CDH{} if and only if $\kappa<\mathfrak{p}$ and in \cite{arh-vm-cdh-cardinality} there is a CH construction of a compact CDH space of uncountable weight that is almost Luzin (that is, every nowhere dense subset is second countable). Finally, a compact \CDH{} space of uncountable weight was constructed without further set-theoretical assumptions in \cite{hg-hr-vm}. 

However, the example in \cite{hg-hr-vm} is the only known example of a compact \CDH{} space of uncountable weight in ZFC. Thus, it is desirable to find other examples of such spaces with different topological properties. For example, it is still unknown if there exist compact \CDH{} spaces of uncountable weight that are either connected or of weight equal to exactly $\mathfrak{c}$ in ZFC (see the open problems at the end of \cite{hg-hr-vm}).

In \cite{hg-arrow} the author showed that if $X$ is a crowded space of countable $\pi$-weight such that ${}\sp\omega{X}$ is CDH, then $X$ contains a copy of the Cantor set. Notice that the Sorgenfrey line is an example of a crowded space of countable $\pi$-weight that is CDH but does not contain Cantor sets. However, it is easy to see that all known examples of infinite compact CDH spaces have topological copies of the Cantor set.

Shortly after the results of \cite{hg-hr-vm} and \cite{hg-arrow} were obtained, Michael Hru\v s\'ak and I were having this same discussion and we considered the following problem.

\begin{ques}\label{mainques}
 Does there exist a crowded, compact, Hausdorff space that is CDH and does not contain topological copies of the Cantor set?
\end{ques}

Arhangel'sk\u\i{} and van Mill constructed their CH example from \cite{arh-vm-cdh-cardinality} using an inverse limit of Cantor sets of length $\omega_1$. So it is natural to try to mimic this technique to answer Question \ref{mainques} under CH. The technical problem faced then is how to kill all possible Cantor sets in $\omega_1$ steps. Back in 2013, Hru\v s\'ak gave such an argument, answering Question \ref{mainques} in the affirmative, by using the guessing principle $\diamondsuit$.

Since that time, it was the intention of the author to answer Question \ref{mainques} in the affirmative assuming only CH. This paper finally provides a complete argument:

\begin{thm}\label{main}
CH implies that there is a first countable, hereditarily separable and $0$-dimensional compact Hausdorff space of uncountable weight that is CDH but does not contain topological copies of the Cantor set. 
\end{thm}

In exchange of using $\diamondsuit$, the proof of Theorem \ref{main} given here will require Lemma \ref{thelemma} below, which is a dynamical property of the group $\homeo{\csf}$ of autohomeomorphisms of the Cantor set.

The paper is organized as follows. Section \ref{section-prelim} contains preliminaries. In Section \ref{section-main} we give the proof of Theorem \ref{main} assuming the truth of Lemma \ref{thelemma}, the proof of which is long and will have to wait for Section \ref{section-lemma}. In Section \ref{section-groups} we give an unexpected application of Lemma \ref{thelemma} and some other related applications. Finally, in Section \ref{section-final}, we make some further remarks about Theorem \ref{main}.

\section{Notation and preliminaries}\label{section-prelim}

If $f:A\to B$ is a function and $C\subset B$, $f\sp\leftarrow[C]$ is the inverse image of $C$. For any set $A$, $\lvert A\rvert$ denotes the cardinality of $A$. If $A$ is a subset in a space $X$, $\inte{A}$, $\cl{A}$, $\bd{A}$ denote the interior, closure, boundary, respectively, of $A$. A crowded space is one without isolated points.

A function $f:X\to Y$ between topological spaces is said to be irreducible if it is closed and for every closed $A$, $f[A]=Y$ if and only if $A=X$. If $X$ is a topological space, let $\co{X}$ denote its set of clopen subsets, this is a Boolean algebra with the inclusion order.

We leave the proof of the following two easy facts to the reader.

\begin{lemma}\label{1stctble}
Let $f:X\to Y$ be a continuous and irreducible function between $0$-dimensional compact Hausdorff spaces. If $y\in Y$ is such that $f\sp\leftarrow(y)=\{x\}$ then $\{f\sp\leftarrow[V]:V\textrm{ is clopen and }y\in V\}$ is a local base in $x$.
\end{lemma}

\begin{lemma}\label{hersep}
Let $X$ be an arbitrary space, $Y$ be hereditarily separable and $f:X\to Y$ be a continuous and irreducible function such that $f\sp\leftarrow(y)$ is finite for each $y\in Y$. Then $X$ is also hereditarily separable.
\end{lemma}

Let $X$ be any topological space. Then $\homeo{X}$ will denote the group of all autohomeomorphisms of $X$. The neutral element of $\homeo{X}$ is of course the identity function $\id_{X}:X\to X$. If $X=\csf$, the identity will be simply denoted by $\id$. Given $S\subset\homeo{X}$, $\gen{S}$ will denote the subgroup of $\homeo{X}$ generated by the set $S$. A set $A\subset X$ is said to be invariant under some subgroup $\G\subset\homeo{X}$ if $h[A]=A$ for every $h\in \G$.

Now, assume that $X$ is a compact metric space with some fixed metric $\rho$. Then $\homeo{X}$ is a topological group with the topology of uniform convergence. Given $f,g\in\homeo{X}$, let $\rho(f,g)=\sup\{\rho(f(x),g(x)):x\in X\}$ and $\sigma_\rho(f,g)=\max\{\rho(f,g),\rho(f\sp{-1},g\sp{-1})\}$. Then $\sigma_\rho$ is a compatible complete metric for $\homeo{X}$. This discussion can be found in \cite{vminf89}.

For any $h\in\homeo{X}$, the set of \emph{fixed points} of $h$ will be denoted by
 $$
 \fix{h}=\{x\in X:h(x)=x\}\text{.}
 $$
 Notice that $\fix{h}$ is always closed. If $\G\subset\homeo{X}$ is a subgroup, let 
 $$
 \fix{\G}=\{x\in X:h(x)=x\textrm{ for some }h\in \G\setminus\{\id_{X}\}\}\text{.}
 $$

 We will call a subgroup $\G\subset\homeo{X}$  \emph{cofinitary} if $\fix{h}$ is finite for every $h\in \G\setminus\{\id_{X}\}$. Cofinitary groups have been considered in the context of permutation groups of the natural numbers and almost disjoint families (see for example the surveys \cite{cofin-1996} and \cite{cofin-2012}).

 We will assume the reader's familiarity with inverse sequences and inverse limits (of length an arbitrary ordinal). See \cite{chigogidze} or \cite[2.5]{eng} for introductions in the general setting.

We will write $\pair{X_\alpha,f_\alpha\sp\beta,\lambda}$ for an inverse sequence of length the limit ordinal $\lambda$, where $X_\alpha$ are the base spaces and $f_\alpha\sp\beta:X_\beta\to X_\alpha$ are the bonding functions. The inverse limit will be written as $\lim_{\leftarrow}{\pair{X_\alpha,f_\alpha\sp\beta,\lambda}}=\pair{X,\pi_\alpha}_{\lambda}$ and consists on the limit space $X$ and a projection $\pi_\alpha:X\to X_\alpha$ for each $\alpha<\lambda$. Concretely, in this situation the limit space may be constructed as
$$
X=\left\{x\in\prod\{X_\alpha:\alpha<\lambda\}:\forall\alpha<\beta<\lambda\ [x(\alpha)=f_\alpha\sp\beta(x(\beta))]\right\},
$$
and the projections are the corresponding restrictions of projections to the factor spaces of the product. An inverse sequence $\pair{X_\alpha,f_\alpha\sp\beta,\lambda}$ is continuous if every time $\gamma<\lambda$ is a limit ordinal, then $\pair{X_\gamma,f_\alpha\sp\gamma}_\gamma=\lim_{\leftarrow}{\pair{X_\alpha,f_\alpha\sp\beta,\gamma}}$. The following result is well-known.

\begin{lemma}\label{limits}
Let $\lambda$ be a limit ordinal and $\pair{X_\alpha,f_\alpha\sp\beta,\lambda}$ be an inverse sequence with $\lim_{\leftarrow}{\pair{X_\alpha,f_\alpha\sp\beta,\lambda}}=\pair{X,\pi_\alpha}_{\lambda}$.
\begin{itemize}
\item[(i)] The set $\{\pi_\alpha\sp\leftarrow[U_\alpha]:\alpha<\lambda,U_\alpha\textrm{ open in }X_\alpha\}$ is a base for the topology of $X$.
\item[(ii)] If $Y\subset X$, then $\lim_{\leftarrow}{\pair{\pi_\alpha[Y],f_\alpha\sp\beta\!\!\restriction_{\pi_\beta[Y]},\lambda}}=\pair{Y,\pi_\alpha\!\!\restriction_Y}_{\lambda}$.
\item[(iii)] If $X_\alpha$ is compact Hausdorff for each $\alpha<\lambda$ and $A$, $B$ are closed subsets of $X$ such that $A\cap B=\emptyset$, then there exists $\alpha<\lambda$ such that $\pi_\alpha[A]\cap\pi_\alpha[B]=\emptyset$.
\item[(iv)] If for each $\alpha<\lambda$, $h_\alpha:X_\alpha\to X_\alpha$ is a homeomorphism and $h_\alpha\circ f_\alpha\sp\beta=f_\alpha\sp\beta\circ h_\beta$ for each $\alpha<\beta<\lambda$, then there is a homeomorphism $h:X\to X$ such that $h_\alpha\circ\pi_\alpha=\pi_\alpha\circ h$ for each $\alpha<\lambda$.
\end{itemize}
\end{lemma}

The following well-known lemma implies that every Cantor set in a long inverse limit must appear in an intermediate step.

\begin{lemma}\label{trappingCantor}
Let $\pair{X_\alpha,\pi_\alpha\sp\beta,\omega_1}$ be an inverse sequence of second countable compact Hausdorff spaces with surjective bounding functions such that its inverse limit is second countable. Then there exists $\gamma<\omega_1$ such that whenever $\gamma<\alpha<\omega_1$ then $\pi_\gamma\sp\alpha$ is a homeomorphism.
\end{lemma}
\begin{proof}
Let $\lim_{\leftarrow}{\pair{X_\alpha,\pi_\alpha\sp\beta,\omega_1}}=\pair{X,\pi_\alpha}_{\omega_1}$. By Lemma \ref{limits}, there are $\{\beta_n:n<\omega\}\subset\omega_1$ and open sets $U_n\subset X_{\beta_n}$, for each $n<\omega$, such that $\{\pi_{\beta_n}\sp\leftarrow[U_n]:n<\omega\}$ is a base for the topology of $X$. Let $\gamma=\sup\{\beta_n:n<\omega\}<\omega_1$, we next argue that this ordinal witnesses the property we want.

Let $\gamma<\alpha<\omega_1$ and let us assume that $\pi_\gamma\sp\alpha$ is not a homeomorphism. So there are $x,y\in X_\alpha$ with $x\neq y$ and $\pi_\gamma\sp\alpha(x)=\pi_\gamma\sp\alpha(y)$. Clearly the collection $\{(\pi_{\beta_n}\sp\alpha)\sp\leftarrow[U_n]:n<\omega\}$ is a base for the topology of $X_\alpha$ so there is $m<\omega$ such that $x\in (\pi_{\beta_m}\sp\alpha)\sp\leftarrow[U_m]$ but $y\notin (\pi_{\beta_m}\sp\alpha)\sp\leftarrow[U_m]$. However, $(\pi_\gamma\sp\alpha)\sp\leftarrow\big[(\pi_{\beta_m}\sp\gamma)\sp\leftarrow[U_m]\big]=(\pi_{\beta_m}\sp\alpha)\sp\leftarrow[U_m]$ contradicts the fact that $\pi_\gamma\sp\alpha(x)=\pi_\gamma\sp\alpha(y)$. Then we have finished the proof.
\end{proof}

\section{Proof of Theorem \ref{main}}\label{section-main}

Theorem \ref{main} will be proved using a continuous inverse limit construction of length $\omega_1$. Let us give an informal picture of how the argument proceeds by describing the construction of the spaces $\{X_\alpha:\alpha<\omega_1\}$ used in the inverse system. The base space $X_0$ is the Cantor set. In each successor step $\alpha+1<\omega_1$, we will choose a Cantor set inside $X_\alpha$ and a point in this Cantor set and split it in two to construct $X_{\alpha+1}$. We always choose different points to split so that at the end we have a $\leq$2-to-1 preimage of $X_0$. By doing this, it follows that all Cantor sets in the limit can be found in intermediate stages and can be destroyed by splitting any of its points.

First, we will show to split a point in an specific Cantor set in each step. This is accomplished by the following result.

\begin{lemma}\label{splitting}
Let $\G$ be a countable subgroup of $\homeo{\csf}$ and let $N\subset\csf$ be a countable set that is invariant under $\G$. Let $Y\subset\csf$ be a closed crowded subspace and $y\in Y\setminus(N\cup\fix{\G})$. Then there exists a continuous and irreducible function $\pi:\csf\to\csf$ and a group monomorphism $m:\G\to\homeo{\csf}$ such that the following conditions hold.
\begin{itemize}
\item[(a)] For all $x\in\csf$, $|\pi\sp\leftarrow(x)|\leq 2$.
\item[(b)] If $x\in N$ then $\pi\sp\leftarrow(x)$ is a singleton.
\item[(c)] The set $\{x\in\csf:|\pi\sp\leftarrow(x)|>1\}$ is countable and equal to $\{h(y):h\in \G\}$.
\item[(d)] The set $\pi\sp\leftarrow[Y]$ is crowded.
\item[(e)] For every $g\in \G$ we have that $\pi\circ m(g)=g\circ \pi$.
\end{itemize}
\end{lemma}

\begin{proof}
The set $\csf\setminus\{y\}$ is the union of a pairwise disjoint family $\{U_n:n<\omega\}$ of non-empty clopen subsets. By recursion it is easy to find an infinite $M\subset\omega$ such that for every $h\in \G$ such that $h(y)\in Y$ both sets $\{n\in M:h[U_n]\cap Y\neq\emptyset\}$ and $\{n\in\omega\setminus M:h[U_n]\cap Y\neq\emptyset\}$ are infinite. So define $A=\{y\}\cup(\bigcup\{U_n:n\in M\})$, this is a regular closed set of $\csf$ the boundary of which is $\{y\}$. Let $\C$ be the smallest subalgebra of regular closed sets of $\csf$ containing $\co{\csf}\cup\{A\}$ and closed under all homeomorphisms from $\G$. 

Notice that $\C$ is countable and atomless so its Stone space $X$ is homeomorphic to the Cantor set. By Stone's duality the function $\pi:X\to\csf$ defined by $\pi(\U)=\bigcap{\U}$ is continuous and irreducible, since it is the dual of the dense inclusion of Boolean algebras $\co{\csf}\hookrightarrow\C$. 

Let $E=\{h(y):h\in \G\}$, this is a countable set disjoint from $N$. To prove conditions (a), (b) and (c) we have to analyze what are the preimages of points in $\csf$ under $\pi$. We will do this by describing bases of ultrafilters in $\C$.

Let $A(0)=A$ and $A(1)=\{y\}\cup(\bigcup\{U_n:n\in \omega\setminus M\})$, which is the complement of $A$ in the algebra of regular closed sets of $X$. Let's prove that following statement holds for all $C\in\C$.
\begin{quote}
$(\ast)_{C}$ If $x\in\bd[{(\csf)}]{C}$, there exists $h\in \G$, $U\in\co{\csf}$ and $i\in 2$ such that $x\in U\cap h[A(i)]= U\cap C$ and $x=h(y)$.
\end{quote}

Notice that $(\ast)_C$ is true for all $C\in\co{X}\cup\{A\}$. Considering all other elements of $\C$ can be obtained by Boolean operations and images under homeomorphisms of $\G$, we shall prove that $(\ast)_C$ holds for all $C\in\C$ in the following steps.

\vskip12pt
\noindent{\it Step 1.}  If $(\ast)_C$ then $(\ast)_{C\sp\prime}$, where $C\sp\prime=\cl[(\csf)]{\csf\setminus C}$ is the complement of $C$ in the algebra of regular closed sets.
\vskip12pt

To prove Step 1, notice that the boundary of a regular closed set is the same as the boundary of its complement (in the algebra of regular closed sets). Let $x\in\bd{C\sp\prime}$, then $x\in\bd{C}$ as well. Thus there are $h\in \G$, $U\in\co{\csf}$ and $i\in 2$ with $x=h(y)$ and $x\in U\cap h[A(i)]=U\cap C$. Notice that $h[A(1-i)]=\{x\}\cup(\csf\setminus h[A(i)])$ because homeomorphisms respect boundaries. Also, since the boundary of $A$ consists on one point only, $\bd[(\csf)]{C}\cap U=\{x\}$. From this considerations it easily follows that $x\in U\cap h[A(1-i)]=U\cap C\sp\prime$.

\vskip12pt
\noindent{\it Step 2.}  If $(\ast)_{C_0}$ and $(\ast)_{C_1}$ then $(\ast)_{C_0\cup C_1}$.
\vskip12pt

For step 2, let $x\in\bd[(\csf)]{C_0\cup C_1}$. First, it may be the case that $x\in C_j\setminus C_{1-j}$ for some $j\in 2$. In this case, there exists $h\in \G$, $U\in\co{\csf}$ and $i\in 2$ such that $x\in U\cap h[A(i)]= U\cap C_j$ and $x=h(y)$. By shrinking if necessary, we may assume that $U\cap C_{i-j}=\emptyset$. Then $x\in U\cap h[A(i)]= U\cap (C_0\cup C_1)$, which implies $(\ast)_{C_0\cup C_1}$.

So assume that $x\in C_0\cap C_1$. Then $x\in\bd[(\csf)]{C_j}$ for $j\in 2$ (otherwise, it would be an interior point) so we have witnesses $h_j\in \G$, $U_j\in\co{\csf}$ and $i(j)\in 2$ such that $x\in U_j\cap h[A(i(j))]= U_j\cap C_j$ and $x=h_j(y)$ for $j\in 2$. Since $x=h_0(y)=h_1(y)$, $h_0\circ h_1\sp{-1}\in \G$ has $y$ as a fixed point. Then by our hypothesis, $h_0=h_1$. 

Let $V=U_0\cap U_1$, then $V\cap h[A(i(0)]=V\cap C_0$ and $V\cap h[A(i(1)]=V\cap C_1$. So there are two cases: $i(0)=i(1)$ or $i(0)+i(1)=1$, let us see that the second case is impossible. If we had $i(0)+i(1)=1$, then since $A(0)\cup A(1)=\csf$ we would have $V=(V\cap C_0)\cup (V\cap C_1)$. So $x\in V\subset C_0\cup C_1$ and $x$ would be an interior point of $C_0\cup C_1$, a contradiction. Thus, we have $i(0)=i(1)$.

From the discussion above it is then clear that $x\in V\cap h_0[A(i(0))]=V\cap (C_0\cup C_1)$.

\vskip12pt
\noindent{\it Step 3.}  If $(\ast)_{C}$ and $h\in \G$ then $(\ast)_{h[C]}$.
\vskip12pt

For step 3, assume that $x\in\bd[(\csf)]{h[C]}$. Since homeomorphisms preserve boundaries, $h\sp{-1}(x)\in\bd[(\csf)]{C}$ so let $h\sp\prime\in \G$, $U\in\co{\csf}$ and $i\in 2$ such that $h\sp{-1}(x)\in U\cap h\sp\prime[A(i)]= U\cap C$ and $h\sp{-1}(x)=h\sp\prime(y)$. Let $V=h[U]\in\co{\csf}$ and $h\sp{\prime\prime}=h\circ h\sp\prime\in \G$, by applying $h$ on the previous equation we obtain $x\in V\cap h\sp{\prime\prime}(A(i))= V\cap h[C]$.
\vskip12pt

Thus, $(\ast)_C$ is true for all $C\in\C$.\vskip10pt 

Let $\U\in X$ and $x=\pi(\U)$. Using $(\ast)_C$ it follows that $\U$ has a base of one of the following forms. If $x\in E$, then $\{U\cap h[A(i)]:x\in U\}$ is a base of $\U$, where $h$ is the unique element of $\G$ such that $x=h(y)$ and $i\in 2$. If $x\in\csf\setminus E$, then $\{U\in\co{\csf}:x\in U\}$ is a base of $\U$. By fixing $x$, in the first case we see that there are exactly two choices of $\U$, one for each $i\in 2$; in the second case such $\U$ is unique. This completes the proof of (a), (b) and (c).

Next we prove (d). Let $Z=Y\setminus E$, this is a dense subset of $Y$ because $E$ is countable and $Y$ is a Cantor set. Since $\pi$ is one-to-one restricted to $\pi\sp\leftarrow[\csf\setminus E]$ and closed, the restriction $\pi\!\!\restriction_{\pi\sp\leftarrow[Z]}:\pi\sp\leftarrow[Z]\to Z$ is a homeomorphism. Thus, $\pi\sp\leftarrow[Z]$ is crowded. So let $x\in E\cap Y$, by the arguments of the previous paragraph there is $h\in \G$ such that $\{U\cap h[A(i)]:x\in U\}$ for $i\in 2$ generate the elements of $\pi\sp\leftarrow(x)$. By the choice of $A$ above and the fact that $Z$ is dense in $Y$, $(U\cap h[A(i)])\cap Z\neq\emptyset$ for $i\in 2$. Thus, if $\U\in\pi\sp\leftarrow(x)$ then every neighborhood of $\U$ intersects $\pi\sp\leftarrow[Z]$. This proves that $\pi\sp\leftarrow[Y]$ is crowded.

We are now left to prove (e). Each homeomorphism $h$ in $\G$ induces an isomorphism $\hat{h}$ of the Boolean algebra $\co{\csf}$. Moreover, $\hat{h}$ is also an isomorphism of $\C$ by the definition of $\C$. Thus, by Stone duality, $\hat{h}$ induces a homeomorphism $m(h)$. The proof that $\pi\circ m(h)=h\circ \pi$ is standard and we leave it to the reader.
\end{proof}

The hardest part in this argument is to do the splitting and at the same time ensuring that the limit space is \CDH{}. Thus, we would also like to choose homeomorphisms that witness this and make sure that they are preserved in the limit. The precise statement of this is as follows.

\begin{lemma}\label{thelemma}
Let $\G$ be a countable subgroup of $\homeo{\csf}$ and $D,E$ two countable dense subsets of $\csf$. If $\G$ is cofinitary, then there is $H\in\homeo{\csf}$ such that $H[D]=E$ and $\gen{\G\cup\{H\}}$ is also cofinitary.
\end{lemma}

The proof of Lemma \ref{thelemma} is hard, mainly because the chosen homeomorphisms are required to form a cofinitary group. We need this because by Lemma \ref{splitting}, in each successor step we are to choose a point $y\in Y$ not in $\fix{\G}$, where $\G$ is the group of homeomorphisms chosen up to that step, so $\fix{\G}$ shouldn't be able to cover the Cantor set $Y$. We will leave the proof of this result for the next chapter and we will concentrate on the remaining of the proof.

\begin{proof}[Proof of Theorem \ref{main}]
The space we are looking for will be constructed as an inverse limit of a sequence $\pair{X_\alpha,\rho_\alpha\sp\beta,\omega_1}$ where $X_\alpha=\csf$ for each $\alpha<\omega_1$.

Notice that an inverse system of Cantor sets of length a countable limit ordinal always gives the Cantor set as an inverse limit. Thus, we only need to specify the bonding functions $\rho_\alpha\sp{\alpha+1}:X_{\alpha+1}\to X_\alpha$ for $\alpha<\omega_1$ by using Lemma \ref{splitting}. These functions will tell us how to split points in order to destroy the Cantor sets. 

In each step $\alpha<\omega_1$ we shall define a countable set $N_\alpha\subset X_\alpha$ to keep track of which points are not to be split in former steps. We also need to know which Cantor set to destroy in each step, this will be done by choosing a Cantor set $F_\alpha\subset X_\alpha$ for each $\alpha<\omega_1$.

We will also need to construct a countable subgroup $\G_\alpha\subset\homeo{X_\alpha}$ for each $\alpha<\omega_1$ and a group monomorphism $m_\alpha\sp\beta:\G_\alpha\to \G_\beta$ for each $\alpha<\beta<\omega_1$. Using this homeomorphisms we will prove that the limit is \CDH{}.

We already know that every space in the sequence is a Cantor set, so for each $\alpha<\omega_1$ let us give enumerations $\{D\pair{\alpha,\beta}:\beta<\omega_1\}$ of all countable dense subsets of $X_\alpha$ and $\{Y\pair{\alpha,\beta}:\beta<\omega_1\}$ of all Cantor sets contained in $X_\alpha$. Let $e:\omega_1\to\omega_1\times\omega_1$ be a bijection such that if $\alpha<\omega_1$ and $e(\alpha)=\pair{\beta,\gamma}$, then $\beta\leq\alpha$. Let $D_0=D(e(0))\subset X_0$.

Our construction will have the following properties:
\begin{itemize}
\item[(1)] $N_0=D_0$;
\item[(2)] for each $\alpha<\omega_1$ and $h\in \G_\alpha$, $h[N_\alpha]=N_\alpha$;
\item[(3)] if $\alpha<\beta<\omega_1$ then
\begin{itemize}
\item[(a)] $\rho_\alpha\sp\beta$ is continuous and irreducible (thus, onto),
\item[(b)] for each $x\in X_\alpha$, $|(\rho_\alpha\sp\beta)\sp\leftarrow(x)|\leq 2$,
\item[(c)] $\{x\in X_\alpha:|(\rho_\alpha\sp\beta)\sp\leftarrow(x)|=2\}$ is countable, and
\item[(d)] $N_\alpha\subset\{x\in X_\alpha:|(\rho_\alpha\sp\beta)\sp\leftarrow(x)|=1\}$;
\end{itemize}
\item[(4)] if $\alpha<\omega_1$ and $e(\alpha)=\pair{\beta,\gamma}$, then there exists $h\in \G_{\alpha+1}$ such that\\ \mbox{$h[(\rho_\beta\sp{\alpha+1})\sp\leftarrow[D(\beta,\gamma)]]=(\rho_0\sp{\alpha+1})\sp\leftarrow[D_0]$};
\item[(5)] if $\alpha<\omega_1$ and $e(\alpha)=\pair{\beta,\gamma}$, then $(\rho_\beta\sp{\alpha+1})\sp\leftarrow[D(\beta,\gamma)]\subset N_{\alpha+1}$;
\item[(6)] if $\alpha<\beta<\omega_1$, then $(\rho_\alpha\sp\beta)\sp\leftarrow[N_\alpha]\subset N_\beta$;
\item[(7)] if $\alpha<\omega_1$, the group $\G_\alpha$ is cofinitary;
\item[(8)] if $\alpha<\beta<\gamma<\omega_1$ then $m_\alpha\sp\gamma=m_\beta\sp\gamma\circ m_\alpha\sp\beta$; and
\item[(9)] if $\alpha<\omega_1$ and $e(\alpha)=\pair{\beta,\gamma}$ then
\begin{itemize}
\item[(a)] $F_\alpha$ is a Cantor set contained in $X_\alpha$,
\item[(b)] $F_\alpha\subset(\rho_\beta\sp\alpha)\sp\leftarrow[Y(\beta,\gamma)]$,
\item[(c)] $(\rho_\beta\sp\alpha)\sp\leftarrow[Y(\beta,\gamma)]\setminus F_\alpha$ is countable,
\item[(d)] $(\rho_\alpha\sp{\alpha+1})\sp\leftarrow[F_\alpha]$ is crowded, and
\item[(e)] there is $x\in F_\alpha$ such that $|(\rho_\alpha\sp{\alpha+1})\sp\leftarrow(x)|=2$.
\end{itemize}
\end{itemize}

Let us describe how to carry out this construction in step $\lambda<\omega_1$. For $\lambda=0$, define $N_0=D_0$ (which is property (1)) and $\G_0=\{\id\}$.

Now consider the case when $\lambda$ is a limit ordinal. As mentioned before, the inverse limit of countably many Cantor sets is a Cantor set so there are continuous functions $\rho_\alpha\sp\lambda:X_\lambda\to X_\alpha$ for $\alpha<\lambda$ such that $\pair{X_\lambda,\rho_\alpha\sp\lambda}_\lambda=\lim_{\leftarrow}{\pair{X_\alpha,\rho_\alpha\sp\beta,\lambda}}$. By (iii) in Lemma \ref{limits} it is easy to see that $\rho_\alpha\sp\lambda$ is irreducible for all $\alpha<\lambda$.

Define $N_\lambda=\bigcup\{(\rho_\alpha\sp\lambda)\sp\leftarrow[N_\alpha]:\alpha<\lambda\}$. By property (3d) it is easy to see that $|(\rho_\alpha\sp\lambda)\sp\leftarrow(x)|=1$ when $\alpha<\lambda$ and $x\in N_\alpha$. Thus, $N_\lambda$ is a countable set.

Using (iv) in Lemma \ref{limits}, it is not hard to define group monomorphisms $m_\alpha\sp\lambda:\G_\alpha\to\homeo{X_\lambda}$ such that condition (8) holds for $\gamma=\lambda$. Then define $\G_\lambda=\bigcup\{m_\alpha\sp\lambda[\G_\alpha]:\alpha<\lambda\}$, which is a group already (that is, we do not have to take the generated group) and clearly countable. To see that $\G_\lambda$ is cofinitary, let $h\in \G_\lambda$. So there is $\alpha<\lambda$ and $g\in \G_\alpha$ such that $h=m_\alpha\sp\lambda(g)$. If $x\in\fix{\{h\}}$, then clearly $\rho_\alpha\sp\lambda(x)\in\fix{\{g\}}$. This means that $\fix{\{h\}}\subset(\rho_\alpha\sp\lambda)\sp\leftarrow[\fix{\{g\}}]$, which is clearly a finite set because preimages of points under $\rho_\alpha\sp\lambda$ are finite.

We will leave the verification of the rest of the properties for this step to the reader. Next we do the successor case, so assume that $\lambda=\eta+1$ and let $e(\eta)=\pair{\alpha,\beta}$. 

In this step we would like to destroy $Y(\alpha,\beta)$, but this is a subset of $X_\alpha$ and we want to split a point in $X_\eta$. Notice that it is possible that $(\rho_\alpha\sp\eta)\sp\leftarrow[Y(\alpha,\beta)]$ contains isolated points so we have to choose one that is not isolated. The set $Z=\{x\in Y(\alpha,\beta):|(\rho_\alpha\sp\eta)\sp\leftarrow(x)|=1\}$ is a cocountable subset of $Y(\alpha,\beta)$ by property (3) so it is crowded. Notice that the function $\rho_\alpha\sp\lambda\!\!\restriction_{(\rho_\alpha\sp\lambda)\sp\leftarrow[Z]}:(\rho_\alpha\sp\lambda)\sp\leftarrow[Z]\to Z$ is closed, one-to-one and continuous so it is a homeomorphism. Thus, $(\rho_\alpha\sp\lambda)\sp\leftarrow[Z]$ is crowded. Let $F_\eta=\cl[X_\eta]{(\rho_\alpha\sp\lambda)\sp\leftarrow[Z]}$, this is the Cantor set we will destroy. Notice that properties (9a), (9b) and (9c) hold for this choice.

Let $y\in X_\eta\setminus(\fix{\G_\eta}\cup N_\beta)$, by Lemma \ref{splitting} there exists an irreducible and continuous function $\rho_\eta\sp{\eta+1}:X_{\eta+1}\to X_\eta$ and a group monomorphism $m_\eta\sp{\eta+1}:\G_\eta\to\homeo{X_{\eta+1}}$ with the properties listed in that lemma. Notice that we can now define $\rho_\gamma\sp{\eta+1}=\rho_\eta\sp{\eta+1}\circ\rho_\gamma\sp\eta$ for all $\gamma<\eta+1$ and it is easy to see that properties in (3) hold for all these functions. Also, by the conditions in Lemma \ref{splitting}, we obtain properties (9d) and (9e).

Consider the sets $D=(\rho_0\sp{\eta+1})\sp\leftarrow[D_0]$ and $D\sp\prime=(\rho_\alpha\sp{\eta+1})\sp\leftarrow[D(e(\alpha))]$. Both are countable sets by property (3c) and dense because the functions considered are irreducible. Moreover, $\rho_0\sp{\eta+1}\!\!\restriction_{D\sp\prime}:D\sp\prime\to D_0$ is one-to-one. By Lemma \ref{thelemma}, there exists $h\in\homeo{X_{\eta+1}}$ such that $h[D]=D\sp\prime$ and $\gen{m_\eta\sp{\eta+1}[\G_\eta]\cup\{h\}}$ is cofinitary. 

Define $\G_{\eta+1}=\gen{m_\eta\sp{\eta+1}[\G_\eta]\cup\{h\}}$. Let $N_{\eta+1}$ be the smallest set containing the set $(\rho_\eta\sp{\eta+1})\sp\leftarrow[N_\eta\cup\{y\}]\cup D_0$ and closed under $\G_{\eta+1}$, clearly we obtain a countable set. The rest of the properties in the construction can be easily checked.

So the space we are looking for is the inverse limit of the sequence we are constructing. To be precise, this space is
$$
X=\left\{x\in\prod_{\alpha<\omega_1}{X_\alpha}:\forall\alpha<\beta<\omega_1\ (\pi_\alpha(x)=\rho_\alpha\sp\beta(\pi_\beta(x)))\right\},
$$
where $\pi_\alpha:X\to X_\alpha$ be the projection into the factor $\alpha<\omega_1$. Clearly $X$ is a compact Hausdorff space. To see that $X$ has uncountable weight, use Lemma \ref{trappingCantor}; however, this will also be clear once we prove that $X$ contains no Cantor sets.  

An important property of $X$ that follows from properties (3) and (6) is the following
\begin{quote}
$(\ast)$ For every $x\in X$ there exists an ordinal $\alpha(x)<\omega_1$ such that if $\alpha(x)<\beta<\omega_1$, then $(\pi_\beta)\sp\leftarrow[\pi_\beta(x)]=\{x\}$
\end{quote}
Informally, every point is split into two points in at most one step. In particular, $\pi_0$ has fibers of cardinality at most $2$ so by Lemma \ref{hersep}, $X$ is hereditarily separable. To see that $X$ is first countable in $x\in X$, consider the ordinal $\alpha(x)$ given in $(\ast)$. Using Lemma \ref{1stctble} it is possible to construct a countable base of $X$ at $x$ using the base of $X_{\alpha(x)}$ at $\pi_\alpha(x)$.

Next, let us show that $X$ is \CDH{}. Notice that $D=(\pi_0)\sp\leftarrow[D_0]$ is a countable dense subset of $X$ and $\pi_0\!\!\restriction_D:D\to D_0$ is one-to-one. Let $E$ be any other countable dense subset of $X$. By $(\ast)$, it is possible to find $\alpha<\omega_1$ such that $\pi_\alpha:E\to\pi_\alpha[E]$ is one-to-one. Let $\beta<\omega_1$ be such that $D(e(\beta))=\pi_\alpha[E]$. By property $(4)$ there is $h\in\homeo{X_{\beta+1}}$ such that $h[(\rho_\alpha\sp{\beta+1})\sp\leftarrow[D(e(\beta))]]=(\rho_0\sp{\beta+1})\sp\leftarrow[D_0]$. By using (iv) in Lemma \ref{limits}, it is not hard to see that the homeomorphisms $\{m_{\beta+1}\sp\gamma:\beta<\gamma\}$ induce a homeomorphism $H\in\homeo{X}$ such that $m_{\beta+1}\sp\gamma(h)\circ \pi_\gamma=\pi_\gamma\circ H$ when $\beta<\gamma<\omega_1$. Then it easily follows that $H[E]=D$. Thus, $X$ is \CDH{}.

Finally, we prove that $X$ contains no Cantor sets. Assume that this is not true and there is $Y\subset X$ homeomorphic to $\csf$ and let $Y_\alpha=\pi_\alpha[Y]$ for every $\alpha<\omega_1$. So $\pair{Y,\pi_\alpha\!\!\restriction_{Y}}_{\omega_1}=\lim\pair{Y_\alpha,(\rho_\alpha\sp\beta)\!\!\restriction_{Y_\alpha},\omega_1}$ by (ii) in Lemma \ref{limits}. By Lemma \ref{trappingCantor}, there is $\lambda<\omega_1$ such that $\rho_\lambda\sp\alpha\!\!\restriction_{Y_\alpha}:Y_\alpha\to Y_\lambda$ is a homeomorphism every time $\lambda\leq\alpha<\omega_1$.

Let $\beta<\omega_1$ be such that $e(\beta)=\pair{\lambda,\gamma}$ and $Y_\gamma=Y(\lambda,\gamma)$. By property (9), $F_\beta$ is the biggest Cantor set contained in $(\rho_\lambda\sp\beta)\sp\leftarrow[Y_\lambda]$ so $Y_\beta\subset F_\beta$. Now take any open set $V$ that intersects $F_\beta$. By property (3c) there is $x\in V\cap F_\beta$ such that the preimage of $\rho_\lambda\sp\beta(x)$ under $\rho_\lambda\sp\beta$ consists on $x$ alone. By the fact that $\rho_\lambda\sp\beta\!\!\restriction_{Y_\beta}:Y_\beta\to Y_\lambda$ is onto and $\rho_\lambda\sp\beta(x)\in Y_\lambda$, we necessarily have that $x\in Y_\beta$. This proves that $Y_\beta$ is dense in $F_\beta$ so in fact $Y_\beta=F_\beta$.

By property (9d), $(\rho_\beta\sp{\beta+1})\sp\leftarrow[Y_\beta]$ is crowded. We proceed by an argument completely similar to the one in the previous paragraph. Since $\rho_\beta\sp{\beta+1}$ is one-to-one in a cocountable set, and both $(\rho_\beta\sp{\beta+1})\sp\leftarrow[Y_\beta]$ and $Y_{\beta+1}$ map to $Y_\beta$ under $\rho_\beta\sp{\beta+1}$, it can be easily proved that $(\rho_\beta\sp{\beta+1})\sp\leftarrow[Y_\beta]=Y_{\beta+1}$. But then property (9e) contradicts the straightforward fact that $\rho_\beta\sp{\beta+1}:Y_{\beta+1}\to Y_\beta$ is a homeomorphism.

The contradiction we have arrived to shows that there are indeed no Cantor sets in $X$. This finishes the proof of the theorem.
\end{proof}

\section{Proof of Lemma \ref{thelemma}}\label{section-lemma}

In this section, we will give a detailed proof of our Lemma \ref{thelemma} that helps extend cofinitary groups. We start with some general facts about groups of homeomorphisms.

 \begin{sclaim}\label{near_no_fixed}
  Let $\pair{X,\rho}$ be a compact metric space, $A\subset X$ a closed subset and $f\in\homeo{X}$. Assume that $A\cap\fix{f}=\emptyset$. Then there exists $\epsilon>0$ such that if $g\in\homeo{X}$ and $\sigma_\rho(f,g)<\epsilon$, then $A\cap\fix{g}=\emptyset$.
 \end{sclaim}
 
 In what follows below, $X=\csf$, $\rho$ will the metric defined by $\rho(x,y)=1/(1+\min\{n<\omega:x(n)\neq y(n)\})$ for $x\neq y$ and we will denote $\sigma=\sigma_\rho$ for this fixed metric $\rho$. We are choosing this metric so that open balls $B(x,\epsilon)=\{y\in\csf:\rho(x,y)<\epsilon\}$, where $x\in\csf$ and $\epsilon>0$, are clopen. We will use the following consequence of fact \ref{near_no_fixed}.
 
 \begin{sclaim}\label{far_preserved}
  Let $U$ and $V$ be clopen subsets of $\csf$ and $f\in\homeo{\csf}$ such that $f[U]=V$. Then there is $\epsilon>0$ such that if $g\in\homeo{\csf}$ is such that $\sigma(f,g)<\epsilon$, then $g[U]=V$.
 \end{sclaim}
 
 In order to prove Lemma \ref{thelemma} we will construct the homeomorphism $H$ in $\omega$ steps. We shall define a Cauchy sequence of homeomorphisms $\{h_n:n<\omega\}\subset\homeo{\csf}$ which will converge to the homeomorphism $H$ we want. That is,
 \begin{itemize}
   \item[(a)] if $m<n<\omega$, $\sigma(h_n,h_m)<\frac{1}{2\sp{m}}$,
 \end{itemize}
 In each step, we will make two promises.
 
  As it is usual in the construction of CDH spaces, we will promise a definition of $H$ restricted to some finite subset of $D$. Let $D=\{d_i:i<\omega\}$ and $E=\{e_i:i<\omega\}$ be enumerations. Thus, in step $n<\omega$ of the construction we will define two finite sets $D_n\in[D]\sp{<\omega}$ and $E_n\in[E]\sp{<\omega}$, and a bijection $\varphi_n:D_n\to E_n$. Then, we will have the following conditions:
 \begin{itemize}
  \item[(b)] for all $n<\omega$, $\{d_i:i< n\}\subset D_{2n}$ and $\{e_i:i<n\}\subset E_{2n+1}$,
  \item[(c)] $\varphi_m\subset\varphi_n$, if $m<n<\omega$, and
  \item[(d)] $\varphi_n\subset h_n$, if $n<\omega$.
 \end{itemize}
 
 The other promise we make in a step is that some element of $\gen{\G\cup\{H\}}$ will only have finitely many fixed points. Thus, we need to enumerate the elements of $\gen{\G\cup\{H\}}$ in advance.
 
 Let $\hat{h}$ be a symbol. We will need to consider the free group generated by $\G$ and $\hat{h}$, which is denoted by $\G[\hat{h}]$ and consists of all non-empty, finite reduced words from the alphabet $(\G\setminus\{\id\})\cup\{\hat{h},\hat{h}\sp{-1}\}$. Here reduced means cancelling $\hat{h}$ and $\hat{h}\sp{-1}$ every time that they are found adjacent. We recall that given any alphabet, there always exists an empty word (different from the empty set), defined to be of length $0$. For our convenience, we will use the empty word in some parts below but we do not include it in the set $\G[\hat{h}]$.
 
 Thus, $\G[\hat{h}]$ can be defined in the following recursive way. First, all elements of $(\G\setminus\{\id\})\cup\{\hat{h},\hat{h}\sp{-1}\}$ are words of length 1. Assume that $\hat{f}$ is a word of length $n<\omega$ and $\hat{f}=\hat{\alpha}\hat{\beta}$, where $\hat{\alpha}$ is a word of length 1 and $\hat{\beta}$ may be the empty word. If $\hat{\alpha}\neq\hat{h}$, then $\hat{h}\sp{-1}\hat{f}$ is a word of length $n+1$. If $\hat{\alpha}\neq\hat{h}\sp{-1}$, then $\hat{h}\hat{f}$ is a word of length $n+1$. Finally, if $\hat{\alpha}\in\{\hat{h},\hat{h}\sp{-1}\}$ and $g\in\G$, then $g\hat{f}$ is a word of length $n+1$.
 
 Given a word $\hat{f}$ of length $n<\omega$, sometimes we will truncate $\hat{f}$ to a certain length. If $1\leq m\leq n$, we will define $\hat{f}_m\in\G[\hat{h}]$ to be the word of length $m$ such that $\hat{f}=\hat{\alpha}\hat{f}_m$ is a reduced expresion. Also, $\hat{f}_0$ will be defined to be the empty word.
 
 Consider $\hat{f}\in\G[\hat{h}]$ and $h\in\homeo{\csf}$. Then $\hat{f}[h]$ will denote the evaluation defined in the obvious way, namely, replace each occurrence of $\hat{h}$ with $h$ and each occurrence of $\hat{h}\sp{-1}$ with $h\sp{-1}$ and evaluate the composition. If $\hat{f}$ is the empty word, then $\hat{f}[h]$ denotes the identity map. We highlight the following observation.
 
 \begin{sclaim}\label{words_continuous}
 For every fixed $\hat{f}$, the map $h\mapsto\hat{f}[h]$ is continuous in the topology of $\homeo{\csf}$. 
 \end{sclaim}
 
 If $\psi$ is any bijection, we can also define $\hat{f}[\psi]$ in a similar way. Notice that in this general case, $\hat{f}[\psi]$ might be the empty function. 
 
 We need another important observation. Let $f,g\in\homeo{\csf}$. Then $\fix{f}=g[\fix{g\sp{-1}\circ f\circ g}]$. This implies that $f$ will have finitely many fixed points if and only if $g\sp{-1}\circ f\circ g$ has finitely many fixed points. Thus, it is not necessary to consider all words of $\G[\hat{h}]$, since it will be enough to check only some in order to obtain a cofinitary generated group.
 
 First, we will define when a word $\hat{f}\notin\G$ is \emph{short}. It is easier to define that $\hat{f}$ is \emph{not} short if $\hat{f}=\hat{\beta}\sp{-1}\hat{\alpha}\hat{\beta}$ in its reduced form, for some word $\hat{\beta}$ of length $1$.  We will also say that $\hat{f}\in \G[\hat{h}]\setminus\G$ is \emph{nice} if it is short and $\hat{f}=\hat{\alpha}\hat{h}$ in its reduced form. The distinction between short and nice will be important. Consider the following operations:
 \begin{enumerate}[label=(\roman*)]
  \item If $\hat{f}=\hat{\beta}\sp{-1}\hat{\alpha}\hat{\beta}$ for some word $\beta$ of length $1$, replace $\hat{f}$ with $\hat{\alpha}$.
  \item If $\hat{f}=g_1\hat{\alpha}g_0$ for $g_0,g_1\in\G$ and $g_0\neq g_1\sp{-1}$, replace $\hat{f}$ with $(g_0\circ g_1)\hat{\alpha}$.
  \item If $\hat{f}=g\hat{\alpha}\hat{h}\sp{-1}$ for some $g\in G$, replace $\hat{f}$ with $\hat{h}\sp{-1}g\hat{\alpha}$.
  \item If $\hat{f}=\hat{h}\sp{-1}\hat{\alpha}g$ for some $g\in G$, replace $\hat{f}$ with its inverse $\hat{f}\sp{-1}=g\sp{-1}\hat{\alpha}\sp{-1}\hat{h}$.
  \item If $\hat{f}=\hat{\alpha}\hat{h}$, do nothing.
 \end{enumerate}
  Let us know describe an algorithm to simplify words. Start with $\hat{f}\in \G[\hat{h}]\setminus\G$. First, do (i) as long as it is possible. Since (i) shortens a word's length by $2$, this has to stop. After we stop, we have arrived to a short word, which satisfies the hypothesis of one of (ii) to (v). Do the corresponding operation. If we are in the hypothesis of (ii), after completing the operation once we will fall into the hypothesis of (iii) or (v). Applying operation (iii) finitely many times leads us to the hypothesis of (iv) and operation (iv) should take us to the hypothesis of (v). 
 
 Given a word $\hat{f}\in \G[\hat{h}]\setminus\G$, the above algorithm allows us to find a nice word $\hat{g}$ with the property that for any $h\in\homeo{\csf}$, $\fix{\hat{f}[h]}$ is finite if and only if $\fix{\hat{g}[h]}$ is finite. In this case, we will say that $\hat{g}$ is a nice word equivalent to $\hat{f}$. If we only apply operation (i) as long as it is possible, then we will say that $\hat{g}$ is a short word equivalent to $\hat{f}$. 
 
 We are ready to give an enumeration of all words that will ultimately represent all elements of $\gen{\G\cup\{H\}}\setminus\G$. Let $\lambda:\omega\to(\G[\hat{h}]\setminus\G)\times\omega$ be an enumeration of all pairs $\pair{\hat{f},n}$, where $\hat{f}$ is a non-empty, reduced, nice word. We will assume that the following property holds:
 \begin{quote}
 $(\ast)$ Let $\lambda(i)=\pair{\hat{f},m}$ and $\lambda(j)=\pair{\hat{g},n}$ with $m\leq n$. Assume that there are reduced words $\hat{f\sp\prime}$ and $\hat{\alpha}$, $\hat{\beta}$ (which may be empty) with $\hat{g}=\hat{\alpha}\hat{f\sp\prime}\hat{\beta}$ such that $\hat{f}$ is a nice word equivalent to $\hat{f\sp\prime}$. Then $i\leq j$.
 \end{quote}
 Notice that $\lambda(0)=\pair{\hat{h},0}$ necessarily.
 
 Given a word $\hat{f}\in\G[\hat{h}]$ we will decide the finite set of fixed points of the evaluation $\hat{f}[H]$ in some step of the recursion. The homeomorphisms $h_n$ will change in every step of the recursion so we have a chance to avoid fixed points by modifying them carefully. However, the functions $\varphi_n$ will fix the value of $H$ at some points from early stages. Here is where some fixed points will be unavoidable. Assume we are in step $n<\omega$ of the construction and $\lambda(n)=\pair{\hat{f},i}$. Then the unavoidable fixed points we are talking about are exactly the fixed points of $\hat{f}[\varphi_n]$, of which there are finitely many (since $\varphi_n$ is finite).
 \begin{itemize}
  \item[(e)] Let $n<\omega$ and $\hat{f}\in\G[\hat{h}]$ such that there is $m\leq n$ with $\lambda(m)=\pair{\hat{f},0}$. Then $\fix{\hat{f}[\varphi_m]}=\fix{\hat{f}[\varphi_n]}$.
  \item[(f)] Let $m\leq n<\omega$ and $\lambda(m)=\pair{\hat{f},i}$. Then $\fix{\hat{f}[h_n]}$ is a subset of the clopen set $\bigcup\{B(x,\frac{1}{i+1}):x\in\fix{\hat{f}[\varphi_m]}\}$.
 \end{itemize}
 Condition (e) says that all fixed points of $\hat{f}$ will be decided in the step where it appears for the first time. Condition (f) is added in order to control fixed points. According to \ref{near_no_fixed} and \ref{words_continuous}, condition (f) implies that $\fix{\hat{f}[H]}$ equals the finite set $\fix{\hat{f}[\varphi_n]}$, where $\lambda(n)=\pair{\hat{f},0}$.

 The last part of our induction hypothesis will be a condition that implies that $H\sp{n}$ has no fixed points for any $n<\omega$.
 \begin{itemize}
  \item[(g)] Let $n<\omega$ and $1\leq \ell<\omega$. Then $\fix{\hat{h}\sp\ell[\varphi_n]}=\emptyset$.
 \end{itemize}
 Strictly speaking, condition (g) is not necessary for our purposes. However, it will help us prove the inductive step.
 
 We have listed all conditions we need for the recursion so next we will describe a step. For $n=0$, take $D_0=E_0=\varphi_0=\emptyset$ and $h_0$ any homeomorphism with no fixed points but $h_0\sp{2}=\id$. 
 
 So now assume that $k<\omega$ and we have defined $h_j$, $D_j$, $E_j$ and $\varphi_j$ for $j\leq k$, satisfying conditions (a) to (f). In step $k+1$ we have to define $h_{k+1}$. What we will do is start with $h_k$ and modify its definition in order to obtain $h_{k+1}$. This modification will be $\epsilon$-close to $h_k$ for some proper $\epsilon>0$.
 
 \begin{claim}\label{closeisenough}
 There is $\epsilon>0$ such that if $\sigma(h_k,h_{k+1})<\epsilon$, then (a) for $n=k+1$ and (f) for $n=k+1$, $m\leq k$ hold.  
 \end{claim}
 
 \begin{proof}[\bf Proof of Claim \ref{closeisenough}]
  
 First, for (a), let $m<k+1<\omega$. We know that $\sigma(h_m,h_k)<\frac{1}{2\sp{m}}$, so $\epsilon$ must be smaller than $\frac{1}{2\sp{m}}-\sigma(h_m,h_k)$ so that $\sigma(h_m,h_{k+1})\leq\sigma(h_m,h_k)+\sigma(h_k,h_{k+1})<\sigma(h_m,h_k)+\epsilon<\frac{1}{2\sp{m}}$.
 
 Now, we turn to (f), where $n=k+1$, $m\leq k$. Let $\lambda(m)=\pair{\hat{g},j}$ and let $U$ be the complement of $\bigcup\{B(x,\frac{1}{i+1}):x\in\fix{\hat{f}[\varphi_m]}\}$. Since $U$ is clopen and $\hat{g}[h_k]$ has no fixed points in $U$, by \ref{near_no_fixed} above, there is $\delta<0$ such that if $h\in\homeo{\csf}$ with $\sigma(\hat{g}[h_k],h)<\delta$, then $\fix{h}\cap U=\emptyset$. By \ref{words_continuous}, there is $\delta\sp\prime>0$ such that if $\sigma(h_k,h)<\delta\sp\prime$, then $\sigma(\hat{g}[h_k],\hat{g}[h])<\delta$. Thus, we have to take $\epsilon<\delta\sp\prime$.
 
 Since the conditions above are only finitely many, we can indeed choose such an $\epsilon>0$ and Claim \ref{closeisenough} is proved.
 \end{proof}

 Here we remark that (f) for $m=n=k+1$ is harder and its proof is part of the work below.
 
 Next, we will define $D_{k+1}$, $E_{k+1}$ and $\varphi_{k+1}$. There are two cases, depending on the parity of $k$. We will assume that $k+1$ is even, the other case can be dealt in an equivalent way. So let $d$ be the element $D\setminus D_k$ with the least subscript and define $D_{k+1}=D_k\cup\{d\}$. It is enough to select $e\in E$ such that $E_{k+1}=E_k\cup\{e\}$ and $\varphi_{k+1}=\varphi_k\cup\{\pair{d,e}\}$, so that condition (b) holds. However, it is possible that some choices of $e$ might violate condition (e). Luckily, the set of elements of $E$ we have to avoid is finite.
 
 \begin{claim}\label{densechoices}
 There is a countable dense set $F\subset E\setminus (D_k\cup E_k)$ such that if we choose $e\in F$ and define $\varphi_{k+1}=\varphi_k\cup\{\pair{d,e}\}$, then condition (b) holds for $k+1=2n$ and condition (e) holds for $n=k+1$.
 \end{claim}
 
 \begin{proof}[\bf Proof of Claim \ref{densechoices}]
  Let $\G_{k+1}$ be the set of all elements $g\in\G$ such that there is $m\leq k+1$, $\hat{f}$, $\hat{\alpha},\hat{\beta}$ words such that $\lambda(m)=\pair{\hat{f},i}$ for some $i<\omega$ and either $\hat{f}=\hat{\alpha}g\hat{\beta}$ or $\hat{f}=\hat{\alpha}g\sp{-1}\hat{\beta}$ is a reduced word. Notice that $\G_{k+1}$ is finite. Then, define 
 $$
 F=E\setminus\{g\sp{-1}(x):x\in D_k\cup\{d\}\cup E_k,g\in\G_{k+1}\cup\{\id\}\}.
 $$
 Clearly, $E\setminus F$ is finite so $F$ is countable dense. It also follows that $F\subset E\setminus (D_k\cup E_k)$. In order to prove Claim \ref{densechoices}, let $e\in F$ and define $\psi=\varphi_k\cup\{\pair{d,e}\}$. 
 
   We only have to prove that given $\hat{f}$, where $\lambda(m)=\pair{\hat{f},0}$ for some $m\leq k$, $\fix{\hat{f}[\varphi_m]}=\fix{\hat{f}[\psi]}$. Let $\ell$ be the length of $\hat{f}$. By induction, $\fix{\hat{f}[\varphi_m]}=\fix{\hat{f}[\varphi_k]}$ so we only need to prove that $\fix{\hat{f}[\varphi_k]}=\fix{\hat{f}[\psi]}$. Notice that since $\varphi_k\subset\psi$, then $\fix{\hat{f}[\varphi_k]}\subset\fix{\hat{f}[\psi]}$ . Recall that since $\hat{f}$ is nice, $\hat{f}_1=\hat{h}$ so $\fix{\hat{f}[\psi]}\subset D_k\cup\{d\}$.
   
   Let $x\in\fix{\hat{f}[\psi]}\setminus\fix{\hat{f}[\varphi_k]}$. Then necessarily $\hat{f}[\varphi_k]$ is not defined at $x$. Let $t$ be the first $1\leq j\leq\ell$ such that $\hat{f}_j[\varphi_k]$ is not defined at $x$. Let $\hat{f}_t=\hat{\alpha}\hat{f}_{t-1}$ be such that $\hat{\alpha}$ is of length $1$ and $\hat{f}_{t-1}$ is of length $t-1$. Also, for every $1\leq j\leq\ell$, let $x_j=\hat{f}_{j}[\psi](x)$. Notice that for $1\leq j<t$, $x_j=\hat{f}_{j}[\varphi_k](x)$.
 
 Notice that $\hat{\alpha}$ cannot be a member of $\G$. To see this, notice that since $\hat{f}_{t-1}[\varphi_k]$ is defined at $x$, $\hat{\alpha}$ is not defined at $x_{t-1}$. The only way this can happen is when $\hat{\alpha}\in\{\hat{h},\hat{h}\sp{-1}\}$. If $\hat{\alpha}=\hat{h}$, then $\varphi_k$ is not defined at $x_{t-1}$ but $\psi$ is; the only way this is possible is if $x_{t-1}=d$ and $x_t=e$. By a similar reasoning, if $\hat{\alpha}=\hat{h}\sp{-1}$, then $x_{t-1}=e$ and $x_t=d$.
 
 \vskip5pt
 \noindent{\it Case 1}: $x_{t-1}=d$, $x_t=e$ and $\hat{\alpha}=\hat{h}$
 \vskip5pt
 
 If $t=\ell$, $\hat{f}_t=\hat{f}$ so $\hat{f}[\psi](x)=e$. Since $e\notin D_k\cup\{d\}$, $x\neq e$ and we obtain a contradiction. Thus, $t<\ell$ must hold. 
 
 Let $\hat{\beta}$ be a word of length $1$ such that $\hat{f}_{t+1}=\hat{\beta}\hat{f}_{t}$. If $\hat{\beta}\in\G$, then by the definition of $F$, $\hat{f}_{t+1}[\psi](x)=\hat{\beta}[\psi](x_t)=\hat{\beta}(x_t)=\notin D_k\cup\{d\}\cup E_k$. Then there are two possibilities next: 
 \begin{itemize}
  \item If $\ell=t+1$, then $x$ is not a fixed point of $\hat{f}[\psi]=\hat{f}_{t+1}[\psi]$, which is a contradiction. 
  \item If $t+1<\ell$, let $\hat{\gamma}$ be a word of length $1$ with $\hat{f}_{t+2}=\hat{\gamma}\hat{f}_{t+1}$. Necessarily, $\hat{\gamma}\in\{\hat{h},\hat{h}\sp{-1}\}$ so $\hat{\gamma}[\psi]$ is not defined in $x_{t+1}$. So $\hat{f}[\psi]$ is not defined in $x$, which is a contradiction.
 \end{itemize}
 The only other possibility is that $\hat{\beta}=\hat{h}$ (otherwise, $\alpha$ and $\beta$ cancel). Since $e\notin D_k\cup\{d\}$ by the definition of $F$, then $\hat{\beta}[\psi]$ is not defined in $e$. So $\hat{f}[\psi]$ is undefined at $x$, which is a contradiction.
 
  \vskip5pt
 \noindent{\it Case 2}: $x_{t-1}=e$, $x_t=d$ and $\hat{\alpha}=\hat{h}\sp{-1}$
 \vskip5pt
 
 First, notice that $t\neq 1$. Otherwise, $\hat{f}_{t-1}[\psi]$ is the identity function and $e=x_{t-1}=\hat{f}_{t-1}[\psi](d)=d$, which contradicts the definition of $F$. So there exists a word $\hat{\beta}$ of length $1$ with $\hat{f}_{t-1}=\hat{\beta}\hat{f}_{t-2}$. 
 
 Clearly, $\hat{\beta}$ cannot be equal to $\hat{h}$ because otherwise, $\hat{\alpha}$ and $\hat{\beta}$ cancel. Moreover, $\hat{\beta}$ cannot be equal to $\hat{h}\sp{-1}$ either. Indeed, $e=\hat{f}_{t-1}[\varphi_k](x)=[\varphi_k\sp{-1}\circ(\hat{f}_{t-2}[\varphi_k])](x)=\varphi_k\sp{-1}(\hat{f}_{t-2}[\psi](x))$ implies that $\varphi_{k}$ is defined at $e$, which contradicts the definition of $F$.
 
 So $\hat{\beta}$ is in $\G$. Now, let us argue that $x_{t-2}\in D_k\cup E_k$, which will be a contradiction by the definition of $F$ and the fact that $\hat{\beta}[\varphi_k](x_{t-2})=e$. If $t-2=0$ and $\hat{f}_{t-2}$ is the empty word, then clearly $x_{t-2}=x\in D_k$. Otherwise, let $\hat{\gamma}$ be a word of length $1$ with $\hat{f}_{t-2}=\hat{\gamma}\hat{f}_{t-3}$. Then $\hat{\gamma}\in\{\hat{h},\hat{h}\sp{-1}\}$ and $\hat{\gamma}[\varphi_m]\in\{\varphi_m,\varphi_m\sp{-1}\}$ so clearly $x_{t-2}\in D_k\cup E_k$. Thus, in this case we also get a contradiction.
 
 Thus, since we obtain a contradiction in all cases considered, we conclude that $\fix{\hat{f}[\varphi_k]}=\fix{\hat{f}[\psi]}$ and Claim \ref{densechoices} is proved.
 \end{proof}

 We promised to define $D_{k+1}$, $E_{k+1}$ and $\varphi_{k+1}$ before the statement of Claim \ref{densechoices}, and we are ready since we only have to choose $e\in F$. Since $F$ is dense, we just choose any $e\in F$ such that $\rho(e,h_k(d))<\epsilon/2$ for $\epsilon>0$ given in Claim \ref{closeisenough}. Define $E_{k+1}=E_k\cup\{e\}$ and $\varphi_{k+1}=\varphi_k\cup\{\pair{d,e}\}$ so that by Claim \ref{densechoices}, conditions (b), (c) and (e) hold.
 
 At this point we have to look at condition (g). We will simply show how the definition of $F$ implies (g) for $n=k+1$ and all $1\leq\ell<\omega$. So assume that $1\leq\ell<\omega$ and there is $x\in\fix{\hat{h}\sp{\ell}[\varphi_{k+1}]}$, we will reach a contradiction. Let $x_0=x$ and $x_j=(\varphi_{k+1})\sp{j}(x)$ for $1\leq j\leq\ell$. Notice that $\{x_j:j\leq\ell\}\subset D_{k+1}$. First, assume that $d\in\{x_j:j\leq\ell\}$. Let $t$ be the first $1\leq j\leq\ell$ such that $x_j=d$. Since $x_0=x_\ell$, $t<\ell$. Thus, $x_{t+1}=e$. However, by the definition of $F$, $e$ is not in $D_{k+1}$, the domain of $\varphi_{k+1}$. This is impossible since any $x_j$ with $j\leq\ell$ is in the domain of $\varphi_{k+1}$. It follows that $\{x_j:j\leq\ell\}\subset D_k$. But this implies that $x\in\fix{\hat{h}\sp{\ell}[\varphi_{k}]}$, which contradicts our inductive hypothesis (g) for $n=k$. Thus, (g) follows for $n=k+1$.
 
 The next step is to construct $h_{k+1}$ itself. We will start with $h_k$ and in a finite sequence of steps, construct homeomorphisms $h_{k+1}=\eta_0,\eta_1,\ldots,\eta_a$ such that $\sigma(\eta_j,\eta_{j+1})<\delta_j$ for some appropriate $\delta_j$. The last homeomorphism constructed will be $\eta_a=h_{k+1}$. The only condition that we need on that sequence of $\delta_j$ is that their sum is $<\epsilon$ so that the hypothesis of Claim \ref{closeisenough} holds. We will not worry about the exact value of $a$ because we can take $\eta_j=1/2\sp{j+1}$ for each $j$.
 
 The first modification we need is a homeomorphism $\eta\in\homeo{\csf}$ with $\sigma(h_k,\eta)<\epsilon/2$, and $\varphi_{k+1}\subset\eta$. Since we chose $e\in F$ such that $\rho(e,h_k(d))<\epsilon/2$, it is possible to modify $h_k$ in a small neighborhood of $d$ that does not intersect $D_k$.

 Next, we need is to modify $\eta$ to a homeomorphism $\eta\sp\prime$ in order that (f) holds for $m=n=k+1$. This is the hardest part of the proof. Further, we need that $\varphi_{k+1}\subset\eta\sp\prime$ so we cannot modify $\eta$ at points of $D_k$. In fact, we prove the following claim, where we bound the set of fixed points of $\hat{f}[\eta\sp\prime]$ for every $f$ that we have considered so far.

  \begin{claim}\label{hardclaim}
 Let $\eta\in\homeo{\csf}$ and $\delta>0$ be such that $\varphi_{k+1}\subset\eta$, $\sigma(\eta,h_k)<\epsilon$. Also, let $\hat{f}$ be such that $\lambda(m)=\pair{\hat{f},i}$ for some $m\leq k+1$ and $i<\omega$. Then, if $U$ is a clopen set of $\csf$ such that $U\cap\fix{\hat{f}[\varphi_{k+1}]}=\emptyset$, there is $\eta\sp\prime\in\homeo{\csf}$ such that $\varphi_{k+1}\subset\eta\sp\prime$, $\sigma(\eta,\eta\sp\prime)<\delta$ and $U\cap\fix{\hat{f}[\eta\sp\prime]}=\emptyset$.
 \end{claim}
 
   \begin{proof}[\bf Proof of Claim \ref{hardclaim}]
    
    We work by induction on $m$. The case $m=0$ immediately holds from the definition of $h_0$ and condition (g). So assume that $m>0$ and thus, $\hat{f}$ has length $\ell\geq 2$.
  
  \vskip5pt
  \noindent{\sffamily Step 1:} Let $\delta\sp\prime>0$. If $x\in D_{k+1}\setminus \fix{\hat{f}[\varphi_{k+1}]}$, then there is $\eta\sp\prime\in\homeo{\csf}$ such that $\varphi_{k+1}\subset\eta\sp\prime$, $\sigma(\eta,\eta\sp\prime)<\delta\sp\prime$ and $x\notin\fix{\hat{f}[\eta\sp\prime]}$.
  \vskip5pt
  
  We only need to prove Step 1 for a fixed such $x\in D_{k+1}\setminus \fix{\hat{f}[\varphi_{k+1}]}$, since there are finitely many points in $D_{k+1}\setminus \fix{\hat{f}[\varphi_{k+1}]}$ and we can appeal to \ref{far_preserved}. That $x\notin\fix{\hat{f}[\varphi_{k+1}]}$ means one of two things. If $\hat{f}[\varphi_{k+1}]$ is defined at $x$ and $\hat{f}[\varphi_{k+1}](x)\neq x$, we just let $\eta\sp\prime=\eta$. Otherwise, $\hat{f}[\varphi_{k+1}]$ is undefined at $x$, we will assume that this is the case.
  
  So define $x_0=x$ and $x_j=\hat{f}_j[\eta](x)$ for $j\leq\ell$. Since $\hat{f}_1=\hat{h}$, $\hat{f}_1[\varphi_{k+1}]=\varphi_{k+1}$ is defined at all points of $D_{k+1}$. Thus, there exists $t$ which is the minimal $j<\ell$ such that $\hat{f}_{j+1}[\varphi_{k+1}]$ is undefined at $x$.  Notice that $x_j=\hat{f}_j[\varphi_{k+1}](x)$ for $j\leq t$. Let $\hat{\alpha}$ be the word such that $\hat{f}_{t+1}=\hat{\alpha}\hat{f}_t$. Then $\hat{\alpha}\notin\G$ since otherwise $\hat{f}_{t+1}$ would be defined at $x$. There are two cases: $\hat{\alpha}$ is either equal to $\hat{h}$ or equal to $\hat{h}\sp{-1}$.
  
  In order to avoid that $x$ is a fixed point of $\hat{f}$, we will proceed as follows: we will either change the definition of $\eta$ in a small neighborhood of $x_t$, if $\hat{\alpha}=\hat{h}$; or change the definition of $\eta\sp{-1}$ in a small neighborhood of $x_t$, if $\hat{\alpha}=\hat{h}\sp{-1}$. This is possible since $x_t\notin D_{k+1}$ if $\hat{\alpha}=\hat{h}$ and $x_t\notin E_{k+1}$ if $\hat{\alpha}=\hat{h}\sp{-1}$. However, first we need to make sure that by changing this definition, we do not loose control of the definition of $\eta$ or $\eta\sp{-1}$ on neighborhoods of points of the form $x_j$ with $j>t$. 
  
  We will assume that $\hat{\alpha}=\hat{h}$. The other case can be treated in an analogous way. Let $\hat{\beta}\in\G[\hat{h}]$ be such that $\hat{f}=\hat{\beta}\hat{f}_t$. Notice that $\hat{\beta}$ has length $\ell-t>0$. Let
  $$
  A=\{j<\ell-t:\hat{\beta}_{j+1}=\hat{h}\hat{\beta}_j\}.
  $$
  Given $j\in A$, it easily follows that $\hat{\beta}_j$ is nice and $x_t\notin\fix{\hat{\beta}_j[\varphi_{k+1}]}$ since $x_t$ is not defined at $\varphi_{k+1}$. By condition $(\ast)$ in the definition of $\lambda$ and our inductive hypothesis, it is easy to argue that there is $\eta_0\in\homeo{\csf}$ such that $\varphi_{k+1}\subset\eta_0$, $\sigma(\eta,\eta_0)<\delta\sp\prime/3$ and $x_t\notin\fix{\hat{\beta}_j[\eta_0]}$ for all $j\in A$. 
  
  Next, consider the set
  $$
  B=\{j<\ell-t:\hat{\beta}_j=\hat{h}\sp{-1}\hat{\beta}_{j-1}\}.
  $$
  We would like to obtain a homeomorphism $\eta\sp{\prime\prime}\in\homeo{\csf}$ such that $\varphi_{k+1}\subset\eta\sp{\prime\prime}$, $\sigma(\eta_0,\eta\sp{\prime\prime})<\delta\sp\prime/3$ and $x_t\notin\fix{\hat{\beta}_j[\eta\sp{\prime\prime}]}$ for all $j\in B$. However, this will not be as easy as in the case for $j\in A$, where it followed by the induction hypothesis in a straightforward manner.
  
  We will use induction on the elements of $B$. So let $j\in B$ and inductively assume that for some $0<\delta\sp{\prime\prime}<\delta\sp\prime/3$ there is $\eta_1\in\homeo{\csf}$ such that $\varphi_{k+1}\subset\eta_1$, $\sigma(\eta_0,\eta_1)<\delta\sp{\prime\prime}$ and $x_t\notin\fix{\hat{\beta}_i[\eta_1]}$ for all $i\in A$ and all $i\in B$ with $i<j$ (if any). Notice that even though $\hat{\beta}_j$ is not short, we can write $\hat{\beta}_j=\hat{h}\sp{-1}\hat{\mu}\sp{-1}\hat{\gamma}\hat{\mu}\hat{h}$ where $\hat{\gamma}$ is a short word (and $\hat\mu$ may be of length $0$).
  
  Choose a clopen set $W$ with the following properties:
  \begin{itemize}
   \item[(1)] $x_t\in W$,
   \item[(2)] $W$ and $\eta_1[W]$ both have diameter $<\delta\sp{\prime\prime}$,
   \item[(3)] $W$ does not intersect $D_{k+1}$, and
   \item[(4)] if $i\in A$ or $i\in B$ and $i$ is at most the length of $\hat{\mu}\hat{h}$, then $(\hat{\mu}\hat{h})_i[\eta_1](x_t)\notin W$.
  \end{itemize}
  
   The set $\fix{\hat\gamma[\varphi_{k+1}]}$ is finite so there is $a\in \eta_1[W]$ such that $b=\hat{\mu}[\eta_1](a)\notin\fix{\hat\gamma[\varphi_{k+1}]}$. Then it is elementary to construct $\eta_2\in\homeo{\csf}$ such that $\eta_2\!\!\restriction_{\csf\setminus W}=\eta_1\!\!\restriction_{\csf\setminus W}$ and $\eta_2(x_t)=a$. Clearly, $\sigma(\eta_1,\eta_2)<\delta\sp{\prime\prime}$ and $\varphi_{k+1}\subset\eta_2$ by properties (2) and (3). Moreover, by property (4) it follows that $b=\hat{\mu}[\eta_2](a)$, which implies that $\hat{\mu}\hat{h}[\eta_2](x_t)\notin\fix{\hat\gamma[\varphi_{k+1}]}$.
  
  The word $\hat\gamma$ might not be nice, but it is equivalent to a nice word $\hat{\gamma}\sp\prime$ by means of the simplification algorithm given above, after applying some instances of operations (ii), (iii) and (iv). Thus, the sets $\fix{\hat\gamma[\varphi_{k+1}]}$ and $\fix{\hat\gamma\sp\prime[\varphi_{k+1}]}$ are related by this algorithm. By property $(\ast)$ of the definition of $\lambda$, we may apply our inductive hypothesis for $\hat{\gamma}\sp\prime$ in any clopen set missing $\fix{\hat\gamma\sp\prime[\varphi_{k+1}]}$.
  
  From these considerations, it is not hard to argue that there exists $\eta_3\in\homeo{\csf}$ with $\varphi_{k+1}\subset\eta_3$, $\sigma(\eta_2,\eta_3)<\delta\sp{\prime\prime}$ and $\hat{\mu}\hat{h}[\eta_3](x_t)\notin\fix{\hat\gamma[\eta_3]}$. This easily implies that $x_t\notin\fix{\hat{\beta}_j[\eta_3]}$. Notice that by choosing $\delta\sp{\prime\prime}$ small enough according to \ref{far_preserved} and \ref{words_continuous}, we obtain that $x_t\notin\fix{\hat{\beta}_i[\eta_3]}$ for all $i\in A$ and all $i\in B$ with $i\leq j$.
  
  Thus, after this procedure, it is possible to obtain $\eta\sp{\prime\prime}$ as desired. Namely, $\varphi_{k+1}\subset\eta\sp{\prime\prime}$, $\sigma(\eta,\eta\sp{\prime\prime})<2\delta\sp\prime/3$ and $x_t\notin\fix{\hat{\beta}_j[\eta\sp{\prime\prime}]}$ for all $j\in A\cup B$. We are ready to construct $\eta\sp\prime$. Similarly as before, choose a clopen set $W$ with the following properties:
  \begin{itemize}
   \item[(1)] $x_t\in W$,
   \item[(2)] $W$ and $\eta\sp{\prime\prime}[W]$ both have diameter $<\delta\sp\prime/3$,
   \item[(3)] $W$ does not intersect $D_{k+1}$, and
   \item[(4)] if $i\in A\cup B$, then $\hat{\beta}_j[\eta\sp{\prime\prime}](x_t)\notin W$.
  \end{itemize}
  Let $\hat{\nu}$ be such that $\hat{\beta}=\hat{\nu}\hat{h}$. Choose $a\in\eta\sp{\prime\prime}[W]$ be any point such that $\hat{\nu}[\eta\sp{\prime\prime}](a)\ne x$. Then consider $\eta\sp\prime\in\homeo{\csf}$ such that $\eta\sp\prime\!\!\restriction_{\csf\setminus W}=\eta\sp{\prime\prime}\!\!\restriction_{\csf\setminus W}$ and $\eta\sp\prime(x_t)=a$. Clearly, (2) and (3) imply that $\sigma(\eta\sp\prime,\eta\sp{\prime\prime})<\delta\sp\prime/3$ and $\varphi_{k+1}\subset\eta\sp\prime$. Then, $\sigma(\eta\sp\prime,\eta)<\delta\sp\prime$. By (4), $\hat{\nu}[\eta\sp{\prime}](a)=\hat{\nu}[\eta\sp{\prime\prime}](a)$ so $\hat{\beta}[\eta\sp\prime](x_t)\ne x$. Moreover, $\hat{f}_t[\eta\sp\prime]=\hat{f}_t[\varphi_{k+1}]$ so $\hat{f}_t[\eta\sp\prime](x)=x_t$. So we obtain that $\hat{f}[\eta\sp\prime](x)\ne x$.
  
  This finishes the proof of Step 1. As mentioned before, by \ref{far_preserved} and \ref{words_continuous} we may assume that $x\notin\fix{\hat{f}[\eta\sp\prime]}$ in fact holds for all $x\in D_{k+1}\setminus \fix{\hat{f}[\varphi_{k+1}]}$. In order to simplify notation, from now on we will assume that the original homeomorphism $\eta$ satisfies the statement of Step 1. 
  
  Next, for each $x\in D_{k+1}\setminus \fix{\hat{f}[\varphi_{k+1}]}$, let $U_x$ be a clopen set such that $x\in U_x$ and $\hat{f}[\eta][U_x]\cap U_x=\emptyset$. Define $V=U\setminus\bigcup\{U_x:x\in D_{k+1}\setminus \fix{\hat{f}[\varphi_{k+1}]}\}$ and consider the following set:
  $$
  A=\{j<\ell:\hat{f}_{j+1}=\hat{h}\hat{f}_j\}.
  $$
  Notice that for all $j\in A$, $\hat{f}_j$ is nice and $V\cap\fix{\hat{f}_j[\varphi_{k+1}]}=\emptyset$. Thus, by our inductive hypothesis we may assume that $V\cap\fix{\hat{f}_j[\eta\sp\prime]}=\emptyset$ for all $j\in A$. Consider now the set
  $$
  B=\{j<\ell:\hat{f}_j=\hat{h}\sp{-1}\hat{f}_{j-1}\}.
  $$
  Just like in the proof of Step 1, we cannot get rid of fixed points of $\hat{f}_j[\eta\sp\prime]$ where $j\in B$ so easily and we need to do some extra work.
  
  Let $j\in B$. Write 
  $$
  \hat{f}_j=(\hat{h})\sp{-1}(\widehat{\mu\sp{j}})\sp{-1}(\widehat{\gamma\sp{j}})(\widehat{\mu\sp{j}})(\hat{h})
  $$
  where $\widehat{\gamma\sp{j}}$ is short. Also, let $V_j=\hat{f}_j[\eta][V]$ and let $\ell_j$ be the length of $\widehat{\mu\sp{j}}\hat{h}$.
  
  \vskip5pt
  \noindent{\sffamily Step 2:} Let $\delta\sp\prime>0$. There is $\eta\sp\prime\in\homeo{\csf}$ with $\varphi_{k+1}\subset\eta\sp\prime$, $\sigma(\eta,\eta\sp\prime)<\delta\sp\prime$; and a clopen set $Z\subset V$ such that, given $j\in B$ and $y\in\fix{\widehat{\gamma\sp{j}}[\varphi_{k+1}]}\cap V_{\ell_j}$, then $y\in(\widehat{\mu\sp{j}}\hat{h})[\eta\sp\prime][Z]$ and $Z\cap\fix{\hat{f}[\eta\sp\prime]}=\emptyset$.
  \vskip5pt
  
  Denote 
  $$
  S=\bigcup\{\fix{\widehat{\gamma\sp{j}}[\varphi_{k+1}]}\cap V_{\ell_j}:j\in B\}
  $$
  and for each $j\in B$, let
  $$
     S_{<j}=\bigcup\{\fix{\widehat{\gamma\sp{i}}[\varphi_{k+1}]}\cap V_{\ell_i}:i\in B,i<j\}.
  $$
  
  The set $Z$ will be constructed by a recursive procedure on $j\in B$. Let us describe our inductive hypothesis next. 
  
  We will assume that there is $\eta_0\in\homeo{\csf}$ such that $\varphi_{k+1}\subset\eta_0$, $\sigma(\eta,\eta_0)<\delta\sp{\prime\prime}$ for some appropriate $\delta\sp{\prime\prime}>0$. Given $i\in B$ with $i<j$ and $y\in \fix{\widehat{\gamma\sp{i}}[\varphi_{k+1}]}\cap V_{\ell_i}$, we shall assume that there are clopen sets $W_{i,y}\sp{0}$, $W_{i,y}\sp{1}$ with the following properties, for $\pi=\eta_0$:
  \begin{enumerate}[label=$(\arabic*)_\pi$]
   \item $y\in(\widehat{\mu\sp{i}}\hat{h})[\pi][W_{i,y}\sp{0}]$,
   \item $W_{i,y}\sp{1}=\hat{f}_i[\pi][W_{i,y}\sp{0}]$, and
   \item if $i\sp\prime\in (A\cup B)\setminus\{i\}$, then $(W_{i,y}\sp{0}\cup W_{i,y}\sp{1})\cap\hat{f}_{i\sp\prime}[\pi][W_{i,y}\sp{0}]=\emptyset$.
  \end{enumerate}
  
  The advantage of properties $(1)_\pi$, $(2)_\pi$ and $(3)_\pi$ is that they are open. That is, since $(1)_{\eta_0}$, $(2)_{\eta_0}$ and $(3)_{\eta_0}$ hold, if $\eta\sp{\prime\prime}$ is close enough to $\eta_0$, then $(1)_{\eta\sp{\prime\prime}}$, $(2)_{\eta\sp{\prime\prime}}$ and $(3)_{\eta\sp{\prime\prime}}$ also hold. Thus, in this step we only need to worry about points in $\fix{\widehat{\gamma\sp{j}}[\varphi_{k+1}]}\cap V_{\ell_j}$ as long as we do small modifications.
  
  In what follows, we will assume that $\fix{\widehat{\gamma\sp{j}}[\varphi_{k+1}]}\cap V_{\ell_j}$ consists of only one point $y_0$. It is not hard to extend this argument to the case when $\fix{\widehat{\gamma\sp{j}}[\varphi_{k+1}]}\cap V_{\ell_j}$ is finite of arbitrary cardinality. Let $x_0$ be the point of $V$ with $(\widehat{\mu\sp{j}}\hat{h})[\eta_0](x_0)=y_0$. 
  
  Notice that $\hat{f}_j[\eta_0](x_0)=x_0$. Notice that $(\widehat{\mu\sp{i}}\hat{h})(x_0)\notin S_{<j}$ because $(3)_{\eta_0}$  implies that there are no fixed points of $\hat{f}_j[\eta_0]$ in the clopen sets $W_{i,y}\sp{0}$ with $i\in B$ and $i<j$. Let $W$ be a clopen set containing $x_0$ and for each $i\in B$, $i\leq j$, let $W_i=\widehat{\mu\sp{i}}\hat{h}[\eta_0][W]$; choose $W$ so that $W_i\cap \fix{\widehat{\gamma\sp{i}}[\varphi_{k+1}]}=\emptyset$ whenever $i\in B$, $i\leq j$. 
  
  We have already explained in Step 1 that regardless of whether $\{\widehat{\gamma\sp{i}}:i\in B\}$ are nice or not, there exists a nice word equivalent to it so we can also use our inductive hypothesis. That is, we may find $\eta_1\in\homeo{\csf}$ such that $\varphi_{k+1}\subset\eta_1$, $\sigma(\eta_0,\eta_1)<\delta\sp{\prime\prime}$ and $W_i\cap \fix{\widehat{\gamma\sp{i}}[\eta_1]}=\emptyset$ for $i\in B$, $i\leq j$. Naturally, we choose $\eta_1$ close enough to $\eta_0$ so that $(1)_{\eta_1}$, $(2)_{\eta_1}$ and $(3)_{\eta_1}$ hold and by \ref{far_preserved}, $W_i=\widehat{\mu\sp{i}}\hat{h}[\eta_1][W]$ for $i\in B$, $i\leq j$.
  
  From these conditions, it follows that there is $x_1\in V\cap W$ with $(\widehat{\mu\sp{j}}\hat{h})[\eta_1](x_1)=y_0$. This point has the additional property that $\hat{f}_i[\eta_1](x_1)\neq x_1$ for $i\in A$ or $i\in B$, $i<j$; and $(\widehat{\mu\sp{i}}\hat{h})(x_1)\notin S_{<j}$. Thus, we may assume that $x_1\notin W_{i,y}\sp{0}\cup W_{i,y}\sp{1}$ for all $y\in S_{<j}$, by shrinking the clopen sets if necessary. 
  
  Let $W_0$ and $W_1$ be clopen sets such that $x_1\in W_0\cap W_1$, $W_1=\hat{f}_j[\eta_1][W_0]$ and $(W_0\cup W_1)\cap (W_{i,y}\sp{0}\cup W_{i,y}\sp{1})=\emptyset$ for all $y\in S_{<j}$. We may also choose $W_0$ in such a way that $W_0\cap\hat{f}_i[\eta_0][W_0]=\emptyset$ whenever $i\in B$ and $i<j$. In fact, by the discussion before the statement of Step 2, let us also choose $W_0$ so that $W_0\cap\hat{f}_i[\eta_0][W_0]=\emptyset$ whenever $i\in A$.
  
  Given $i\in B$ with $i>j$, let $\widehat{\nu\sp{i}}$ be the word such that $\hat{f}_i=\widehat{\nu\sp{i}}\hat{f}_j$. Notice that $\widehat{\nu\sp{i}}$ is short (but not nice). Since the set
  $$
  \bigcup\{\fix{\widehat{\nu\sp{i}}}:i\in B,i>j\}
  $$
  is finite, there is a point $x_2\in W_0\cap W_1$ that misses it and $\rho(x_2,x_1)<\delta\sp{\prime\prime}$. Let $a=\eta_1(x_1)$. So consider $\eta_2\in\homeo{\csf}$ with $\eta_2(x_2)=a$ and $\eta_2$ is equal to $\eta_1$ outside some clopen set $W\subset W_0\cap W_1$ containing $x_1$ and $x_2$ of diameter $<\delta\sp{\prime\prime}$.
  
  Now, by our choice of $W_0$ and $\eta_2$, it is not hard to argue that $\widehat{\mu\sp{j}}[\eta_1](a)=\widehat{\mu\sp{j}}[\eta_2](a)$. Then, it follows that $\widehat{\mu\sp{j}}\hat{h}[\eta_2](x_2)=y_0$ so $\hat{f}_j[\eta_2](x_2)=x_2$. If we choose $W$ small enough (thus, $x_2$ close enough to $x_1$), $\eta_2$ will have the properties of $\eta_1$ that we have mentioned before, with $x_2$ taking the place of $x_1$. Moreover, we have the advantage that $x_2\notin\fix{\widehat{\nu\sp{i}}[\varphi_{k+1}]}$ for $i\in B$ with $i>j$.
  
  By shrinking $W_0$ we may assume that $W_1\cap\fix{\widehat{\nu\sp{i}}[\varphi_{k+1}]}=\emptyset$ for $i\in B$ with $i>j$. By our inductive assumption, there is $\eta\sp{\prime\prime}\in\homeo{\csf}$ such that $\sigma(\eta\sp{\prime\prime},\eta_1)<\delta\sp{\prime\prime}$, $\varphi_{k+1}\subset\eta\sp{\prime\prime}$ and $W_0\cap W_1\cap\fix{\widehat{\nu\sp{i}}}=\emptyset$ for all $i\in B$, $i>j$. Let us define $\eta_3\in\homeo{\csf}$ to be equal to $\eta\sp{\prime\prime}$ for points in $W_0\cap W_1$ and equal to $\eta_2$ otherwise. If $\eta\sp{\prime\prime}$ is close enough to $\eta_2$ so that $\eta\sp{\prime\prime}[W_0\cap W_1]=\eta_2[W_0\cap W_1]$ (by \ref{far_preserved}), $\eta_3$ will be well-defined. 
  
  Let $x_3\in W_0$ be such that $\eta_3(x_3)=a$. Since $\widehat{\mu\sp{j}}[\eta_3](a)=\widehat{\mu\sp{j}}[\eta_2](a)=y$, we obtain that $\widehat{\mu\sp{j}}\hat{h}[\eta_3](x_3)=y_0$ so $\hat{f}_j[\eta_3](x_3)=x_3$ and $x_3\notin\fix{\widehat{\nu\sp{i}}[\eta_3]}$ for $i\in B$ with $i>j$. Finally, shrink $W_0$ so that $x_3\in W_0$ and $W_0\cap\fix{\widehat{\nu\sp{i}}[\eta_3]}=\emptyset$ for all $i\in B$ with $i>j$. Define $W_{j,y_0}\sp{0}=W_0$ and $W_{j,y_1}\sp{0}=W_1$. Then it follows that $(1)_{\eta_3}$, $(2)_{\eta_3}$ and $(3)_{\eta_3}$ hold for $y=y_0$.
  
  So we have finished the recursive construction of the clopen sets $W_{i,y}\sp{0}$ and $W_{i,y}\sp{1}$ for all $y\in\fix{\widehat{\gamma\sp{i}}[\varphi_{k+1}]}\cap V_{\ell_i}$, where $i\in B$. Also, we have a homeomorphism $\eta\sp{\prime\prime\prime}\in\homeo{\csf}$ with $\varphi_{k+1}\subset\eta\sp{\prime\prime\prime}$ such that $(1)_{\eta\sp{\prime\prime\prime}}$, $(2)_{\eta\sp{\prime\prime\prime}}$ and $(3)_{\eta\sp{\prime\prime\prime}}$ hold for all $y\in S$. By choosing $\delta\sp{\prime\prime}$ carefully, we may assume that $\sigma(\eta,\eta\sp{\prime\prime\prime})<\delta\sp\prime/2$. We are thus ready to construct $\eta\sp\prime\in\homeo{\csf}$ required by the statement of Step 2.

  Fix some $j\in B$ and $y\in \fix{\widehat{\gamma\sp{j}}[\varphi_{k+1}]}\cap V_{\ell_j}$. Let $\widehat{\alpha\sp{j}}$ and $\widehat{\beta\sp{j}}$ be the words such that $\hat{f}=(\widehat{\alpha\sp{j}})\sp{-1}(\widehat{\beta\sp{j}})(\widehat{\alpha\sp{j}})(\hat{f_j})$ and $\widehat{\beta\sp{j}}$ is short. We need to do some modifications in order to remove some fixed points from $\widehat{\beta\sp{j}}$. However, we will not be as detailed as before because the arguments are completely analogous. First, we may modify $\eta\sp{\prime\prime\prime}$ inside $W_{j,y}\sp{0}$ so that there is $x\in W_{j,y}\sp{0}\cap W_{j,y}\sp{1}$ with $\widehat{\mu\sp{j}}\hat{h}(x)=y$ and $x\notin \fix{\widehat{\beta\sp{j}}[\varphi_{m}]}$. After this, we may shrink $W_{j,y}\sp{0}$ so that $\widehat{\alpha\sp{j}}[\eta\sp{\prime\prime\prime}][W_{j,y}\sp{1}]\cap\fix{\widehat{\beta\sp{j}}[\varphi_{m}]}=\emptyset$. Then, by another modification of $\eta\sp{\prime\prime\prime}$, we may further assume that $\widehat{\alpha\sp{j}}[\eta\sp{\prime\prime\prime}][W_{j,y}\sp{1}]\cap\fix{\widehat{\beta\sp{j}}[\eta\sp{\prime\prime\prime}]}=\emptyset$. These three steps can be proved in essentially the same way as similar situations before.

  Thus, we may assume that $\eta\sp\prime\in\homeo{\csf}$ is such that $(1)_{\eta\sp{\prime}}$, $(2)_{\eta\sp{\prime}}$ and $(3)_{\eta\sp{\prime}}$ hold for all $y\in S$ and moreover, $\widehat{\alpha\sp{j}}[\eta\sp\prime][W_{j,y}\sp{1}]\cap\fix{\widehat{\beta\sp{j}}[\eta\sp\prime]}=\emptyset$ every time $j\in B$ and $y\in \fix{\widehat{\gamma\sp{j}}[\varphi_{k+1}]}\cap V_{\ell_j}$. 
  
  We are only left to define $Z$. For each $j\in B$ and $y\in \fix{\widehat{\gamma\sp{j}}[\varphi_{k+1}]}\cap V_{\ell_j}$, let $x_{j,y}\in W_{j,y}\sp{0}$ be such that $\widehat{\mu\sp{j}}\hat{h}[\eta\sp\prime](x_{j,y})=y$. Then, $\hat{f}_{j}[\eta\sp\prime](x_{j,y})=x_{j,y}$. Since $x_{j,y}\in W_{j.y}\sp{1}$, then $\widehat{\alpha\sp{j}}[\eta\sp\prime](x_{j,y})$ is not a fixed point of $\widehat{\beta\sp{j}}[\eta\sp\prime]$. Thus, $(\widehat{\alpha\sp{j}})\sp{-1}(\widehat{\beta\sp{j}})(\widehat{\alpha\sp{j}})[\eta\sp\prime](x_{j,y})\neq x_{j,y}$. From this it follows that 
  $\hat{f}[\eta\sp\prime](x_{j,y})\neq x_{j,y}$. So simply define $Z$ to be a clopen set containing the finite set $\{x_{j,y}:j\in B, y\in S\}$ and such that $Z\cap \hat{f}[\eta\sp\prime][Z]=\emptyset$.
  
  This concludes the proof of Step 2. As in the case of Step 1, in what follows we will assume that $\eta$ in fact has the properties in the statement of Step 2. This is done in order to simplify notation. Let $V\sp\prime=V\setminus Z$. For $j\in A$ we already know that $V\sp\prime\cap\fix{\hat{f}_j[\eta]}=\emptyset$, next we would like to obtain this for $j\in B$. 
  
  Fix $j\in B$. By the definition of $Z$, we know that $\widehat{\mu\sp{j}}\hat{h}[\eta][V\sp\prime]$ does not intersect $\fix{\widehat{\gamma\sp{j}}[\varphi_{k+1}]}$. Since $\widehat{\gamma\sp{j}}$ is short, we have already argued that we may apply our inductive hypothesis. Thus, we will assume that $\eta$ is already such that $\widehat{\mu\sp{j}}\hat{h}[\eta][V\sp\prime]$ does not intersect $\fix{\widehat{\gamma\sp{j}}[\eta]}$. This easily implies that $V\sp\prime\cap\fix{\hat{f}_j[\eta]}=\emptyset$.
  
  Thus, we may assume that $V\sp\prime\cap\fix{\hat{f}_j[\eta]}=\emptyset$ for all $j\in A\cup B$. By compactness, there exists a partition $\V$ of $V\sp\prime$ into clopen sets such that every time $W\in\V$ and $j\in A\cup B$, then $W\cap\hat{f}_j[\eta][W]=\emptyset$.
  
  \vskip5pt
  \noindent{\sffamily Step 3:} Let $W\in\V$ and $\delta\sp\prime>0$. Then there exists $\eta\sp\prime\in\homeo{\csf}$ such that $\varphi_{k+1}\subset \eta\sp\prime$, $\sigma(\eta,\eta\sp\prime)<\delta\sp\prime$ and $W\cap \fix{\hat{f}}=\emptyset$.
  \vskip5pt
  
  Let $W\sp\prime=\eta[W]$. We can find a finite partition $W=\bigcup\{W_i:i<t\}$ into clopen sets, where the diameter of both $W_i$ and $W_i\sp\prime=\eta[W_i]$ is less than $\delta\sp\prime$ for all $i<t$. Write $\hat{f}=\hat{\alpha}\hat{h}$. For each $i<t$, let $W\sp\prime_i=C_i\sp{0}\cup C_i\sp{1}$ be a partition into pairwise disjoint clopen sets. Also, for each $i<t$, choose a partition $W_i=W_i\sp{0}\cup W_i\sp{1}$ so that $\hat{\alpha}[\eta][C_i\sp{0}]\subset W_i\sp{1}$ and $\hat{\alpha}[\eta][C_i\sp{1}]\subset W_i\sp{0}$. 
  
  Let us define $\eta\sp\prime\in\homeo{\csf}$ in the following way: $\eta\sp\prime\!\!\restriction_{\csf\setminus W}=\eta\!\!\restriction_{\csf\setminus W}$ and for each $i<t$, $\eta\sp\prime\!\!\restriction_{W_i\sp{0}}:W_i\sp{0}\to C_i\sp{0}$ and $\eta\sp\prime\!\!\restriction_{W_i\sp{1}}:W_i\sp{1}\to C_i\sp{1}$ are arbitrary homeomorphisms. From the definition of the partition of $W$ it follows that $\sigma(\eta,\eta\sp\prime)<\delta\sp\prime$. Also, $\varphi_{k+1}\subset \eta\sp\prime$ because $W\cap D_{k+1}=\emptyset$.
  
  Now, let us see that $\hat{f}[\eta\sp\prime]$ does not have fixed points. First, recall that if $j\in A\cup B$, then $W\cap\hat{f}_j[\eta][W]=\emptyset$. From this, it is not too hard to argue that $\hat{\alpha}[\eta]\!\!\restriction_{W\sp\prime}=\hat{\alpha}[\eta\sp\prime]\!\!\restriction_{W\sp\prime}$. Let $i<t$. Then 
  $$
  \hat{f}[\eta\sp\prime][W_i\sp{0}]=\hat{\alpha}[\eta\sp\prime][C_i\sp{0}]=\hat{\alpha}[\eta][C_i\sp{0}]\subset W_i\sp{1},
  $$
  and it can be proved in an analogous way that $\hat{f}[\eta\sp\prime][W_i\sp{1}]\subset W_i\sp{0}$. Thus, we obtain that $W_i\cap\fix{\hat{f}}=\emptyset$ for all $i<t$. Thus, $W\cap \fix{\hat{f}}=\emptyset$.
  
  This concludes the proof of Step 3. Applying Step 3 for all $W\in\V$, it easily follows that $V\sp\prime\cap\fix{\hat{f}}=\emptyset$. So this finally completes the proof of the claim.
   \end{proof}

   So just choose $U$ to be equal to the complement of $\bigcup\{B(x,\frac{1}{i+1}):x\in\fix{\hat{f}[\varphi_{k+1}]}\}$ and use Claim 3. We immediately obtain the conclusion of (f) for $m=n=k+1$. This concludes the recursive construction. And, as discussed above, this is enough to complete the proof of Lemma \ref{thelemma}.
 
\section{Generalizations about cofinitary groups}\label{section-groups}

We incidentaly obtain as a corollary that for the Cantor set, being CDH is witnessed by a proper subgroup of $\homeo{\csf}$, the elements of which have a special property. We will say that a topological space $X$ is CDH with respect to a group $\G\subset\homeo{X}$ if every time $D,E\subset X$ are countable dense sets, there exists $h\in\G$ such that $h[D]=E$.
 
 \begin{coro}\label{cofin-CDH}
  CH implies that there exists a cofinitary group $\G\subset\homeo{\csf}$ such that $\csf$ is CDH with respect to $\G$.
 \end{coro}
 \begin{proof}
  Enumerate all pairs countable dense subsets in a sequence of length $\omega_1$ and recursively apply Lemma \ref{thelemma}.
 \end{proof}

 So it is natural to ask whether CH is necessary.
 
 \begin{ques}
  Is there a cofinitary group $\G\subset\homeo{\csf}$ such that $\csf$ is CDH with respect to $\G$, in ZFC?
 \end{ques}
 
 It turns out that if we assume Martin's axiom we can prove a version of Lemma \ref{thelemma} which gives us the following result.
 
 \begin{thm}\label{cofin-CDH-MA}
  MA implies that there exists a cofinitary group $\G\subset\homeo{\csf}$ such that $\csf$ is CDH with respect to $\G$.
 \end{thm}
 \begin{proof}
 The proof is analogous to the proof of Corollary \ref{cofin-CDH}: enumerate all pairs of countable dense sets and recursively construct a cofinitary group. Clearly, we need some kind of version of Lemma \ref{thelemma} for groups of cardinality $<\mathfrak{c}$ under MA.
 
 Let $D,E$ be two countable dense sets of $\csf$ and let $\G\subset\homeo{\csf}$ be a cofinitary subgroup of cardinality $<\mathfrak{c}$. We would like to define $H\in\homeo{\csf}$ such that $H[D]=E$ and $\gen{\G\cup\{H\}}$ is cofinitary. We will define a forcing $\mathbb{P}$ and argue that it is ccc. It will remain to use MA to extract a generic subset from $\mathbb{P}$ and use it to define $H$, this part we leave to the reader who can mimic the proof of Lemma \ref{thelemma}.
 
 We shall use all terminology from Section \ref{section-lemma}. Recall that the metric $\rho$ in $\csf$ has its open balls clopen and that $\sigma$ is the induced metric in $\homeo{\csf}$. A set of short words $\mathcal{W}\subset\G[\hat{h}]$ will be called \emph{downwards closed} if every time $\hat{f}\in\mathcal{W}$ and there are words $\hat{g},\hat{\alpha},\hat{\beta}$ such that $\hat{f}=\hat{\alpha}\hat{g}\hat{\beta}$ is a reduced expression and $\hat{g}$ is short, then $\hat{g}\in\mathcal{W}$.
 
 So define $p\in\mathbb{P}$ if and only if $p=\pair{h_p,\varphi_p,n_p,\mathcal{W}_p}$, where:
 \begin{enumerate}
  \item $h_p\in\homeo{\csf}$,
  \item $\varphi_p\subset D\times E$ is a finite bijection,
  \item $\varphi_p\subset h_p$,
  \item for all $1\leq\ell<\omega$, $\fix{\varphi_p\sp\ell}=\emptyset$.
  \item $1\leq n_p<\omega$,
  \item $\mathcal{W}_p$ is a finite set of short words of $\G[\hat{h}]$ that is downwards closed, and
  \item if $\hat{f}\in\mathcal{W}_p$, the set $\fix{\hat{f}[h_p]}$ is a subset of\\ $\bigcup\{B(x,1/n_p):x\in\fix{\hat{f}[\varphi_p]}\}$.
 \end{enumerate}
We define $q\leq p$ in $\mathbb{P}$ if either $p=q$ or the following hold:
\begin{enumerate}[label=(\roman*)]
 \item $n_p<n_q$,
 \item $\varphi_p\subset\varphi_q$,
 \item $\sigma(h_p,h_q)<1/n_p-1/n_q$, and
 \item if $\hat{f}\in\mathcal{W}_p$, then $\fix{\hat{f}[\varphi_p]}=\fix{\hat{f}[\varphi_q]}$.
\end{enumerate}

Naturally, all of these properties have some corresponding statement in Section \ref{section-lemma}. Now we prove that in fact $\mathbb{P}$ is $\sigma$-centered. For this, we have to write $\mathbb{P}$ as a countable union of centered subsets.

For each $p\in\mathbb{P}$, notice that condition (7) is open. This means that there exists $1\leq M_p<\omega$ such that if $\hat{f}\in\mathcal{W}$ and $g\in\homeo{\csf}$ is such that $\sigma(g,h_p)<1/M_p$, then $\fix{\hat{f}[g]}$ is still contained in $\bigcup\{B(x,1/n_p):x\in\fix{\hat{f}[\varphi_p]}\}$.

Fix some countable dense set $\mathfrak{D}\subset\homeo{\csf}$. So for fixed finite $\varphi\subset D\times E$, $n,M<\omega$ and $h\in\mathfrak{D}$, consider the set:
$$
\mathbb{P}\sp\ast=\{p\in\mathbb{P}:M>M_p,\ \sigma(h,h_p)<1/(2M),\ 1/M<1/n-1/(n+1),\ \varphi_p=\varphi,\ n_p=n\}.
$$
Since there are countably many such sets, it is enough to prove that $\mathbb{P}\sp\ast$ is centered.

So let $Q\subset\mathbb{P}$ be a finite set, we need to construct a common extension $r\in\mathbb{P}$. Define $n_r=n+1$ and $\varphi_r=\varphi$. Clearly, these definitions are enough to ensure properties (i), (ii) and (iv) of the extension. Also, let 
$$
\mathcal{W}_r=\bigcup\{\mathcal{W}_q:q\in Q\},
$$
this is clearly a finite set of short words that is downwards closed. 

Finally, choose any $h_r\in\homeo{\csf}$ with $\varphi\subset h_r$ and $\sigma(h,h_r)<1/(2M)$, this is not hard to do. Notice that $h_r$ satisfies property (7) because $M>M_q$ for all $q\in Q$. Also, (iii) holds trivially for all $q\in Q$. Thus, this such constructed $p$ is an element of $\mathbb{P}$ that is a common extension to all elements of $Q$.

Thus, our poset is $\sigma$-centered. It remains to find adequate dense sets that will allow us to construct the desired homeomorphism $H$ using a generic set of $\mathbb{P}$. However, this argument is exactly analogous to the proof of Lemma \ref{thelemma} in Section \ref{section-lemma}. Thus, we will leave this work to the reader.
\end{proof}

 Another similar question is whether fixed points are really necessary in these types of subgroups.

  \begin{ques}
  Let $\G\subset\homeo{\csf}$ be a subgroup such that $\csf$ is CDH with respect to $\G$. Does there exist some $h\in\G$ such that $\fix{h}\neq\emptyset$?
 \end{ques}
 
 Then, we can also ask for similar properties for other CDH spaces.
 
 \begin{ques}
  Let $X$ be the the space of irrationals, a metrizable manifold or the Hilbert cube.
  \begin{itemize}
   \item[(a)] Is there a cofinitary group $\G\subset\homeo{X}$ with $X$ is CDH with respect to $\G$?
   \item[(b)] In case that $X$ does not have the fixed point property, is there a $\G\subset\homeo{X}$ such that $X$ is CDH with respect to $\G$ and $\fix{\G}=\emptyset$?
  \end{itemize}
 \end{ques}

\section{Final remarks about compact CDH spaces}\label{section-final}
 
 The first natural question is if the space constructed in the proof of Theorem \ref{main} can be constructed with no further hypothesis from ZFC.

\begin{ques}
Is it consistent that all infinite compact Hausdorff CDH spaces contain topological copies of the Cantor set?
\end{ques}

From the fact that the ZFC example of a compact CDH space with uncountable weight constructed in \cite{hg-hr-vm} is linearly ordered, a natural question is whether the example constructed in this paper can be linearly ordered. From the proof it seems hard to try to preserve the order relation in the recursive construction.

\begin{propo}
An infinite, linearly ordered, CDH, compact and Hausdorff space must contain topological copies of the Cantor set.
\end{propo}
\begin{proof}
Assume that there exists a space $X$ with the characteristics in the statement of this proposition. First, it is not hard to prove that any \CDH{} space is a topological sum of homogeneous \CDH{} spaces, this can be easily done following the proof of \cite[Theorem, p. 20]{fitz-lauer}. By the well-known characterization of the reals as the only separable, connected, linearly ordered set without endpoints, every non-trivial connected component of $X$ is homeomorphic to $[0,1]$. So we may assume that $X$ is $0$-dimensional and moreover it does not have isolated points.

By a result of Ostaszewki's (\cite{arrows}), it is not hard to see that there exists a set $Y$ dense in $(0,1)$ such that $X$ is homeomorphic to the space $([0,1]\times\{0\})\cup(Y\times\{1\})$ with the topology given by the lexicographic order.

Now let us show that $Y$ has the Baire property. If not, there exists a family $\{C_n:n<\omega\}$ of closed and nowhere dense subsets of $[0,1]$ and an open set $U\subset[0,1]$ such that $Y\cap U\subset\bigcup\{C_n:n<\omega\}$. Then there exists a Cantor set $C\subset U\setminus Y$ and $C\times\{0\}\subset X$ is homeomorphic to the Cantor set.

Using that $Y$ has the Baire property, it is possible to use arguments similar to the ones in Section 3 of \cite{arh-vm-cdh-cardinality} (or Section 3 in \cite{hg-arrow}) that show that $X$ is not \CDH{}. We leave the details to the reader.
\end{proof}

Finally, the techniques we have only produce $0$-dimensional spaces so we ask the following.

\begin{ques}
Is there a connected example of a compact Hausdorff CDH space that contains no Cantor sets?
\end{ques}

\end{document}